\documentclass[a4paper,english,cleveref,thm-restate]{lipics-v2021} 

\usepackage{csquotes}
\usepackage{booktabs}
\usepackage{mathtools}
\usepackage{aligned-overset}
\usepackage[nocompress,noadjust]{cite}
\usepackage{xspace}
\usepackage{bm}
\usepackage{hypbmsec}
\usepackage{fnpct} 
\setfnpct{after-punct-space=-0.15em}

\usepackage[color=white!50!orange,disable]{todonotes}

\newtheorem{question}[theorem]{Question}

\newcommand{\thomas}[1]{\todo[linecolor=blue!25,backgroundcolor=blue!25]{\footnotesize#1}}
\newcommand{\Thomas}[1]{\todo[inline,caption=,backgroundcolor=blue!25]{#1}}

\hideLIPIcs  

\graphicspath{{figures/}}

\title{Product Structure and Treewidth of Hyperbolic Uniform Disk Graphs}

\author{Thomas Bläsius}{Karlsruhe Institute of Technology, Germany}{thomas.blaesius@kit.edu}{https://orcid.org/0000-0003-2450-744X}{}
\author{Emil Dohse}{Karlsruhe Institute of Technology, Germany}{}{}{}
\author{Deborah Haun}{Karlsruhe Institute of Technology, Germany}{deborah.haun@student.kit.edu}{https://orcid.org/0009-0004-2365-0804}{}
\author{Laura Merker}{Karlsruhe Institute of Technology, Germany}{laura.merker2@kit.edu}{https://orcid.org/0000-0003-1961-4531}{}

\authorrunning{T. Bläsius, E. Dohse, D. Haun, L. Merker} 

\Copyright{Thomas Bläsius, Emil Dohse, Deborah Haun, and Laura Merker} 



\category{} 

\relatedversion{An extended abstract of this paper is published in the Proceedings of the 42nd International Symposium on Computational Geometry (SoCG 2026)}

\acknowledgements{%
    We thank Torsten Ueckerdt and Marcus Wilhelm for fruitful discussions.
    Funded by the Deutsche Forschungsgemeinschaft (DFG, German Research Foundation) -- Projektnummer 524989715
}

\nolinenumbers 

\EventEditors{Hee-Kap Ahn, Michael Hoffmann, and Amir Nayyeri}
\EventNoEds{3}
\EventLongTitle{42nd International Symposium on Computational Geometry (SoCG 2026)}
\EventShortTitle{SoCG 2026}
\EventAcronym{SoCG}
\EventYear{2026}
\EventDate{June 2--5, 2026}
\EventLocation{New Brunswick, NJ, USA}
\EventLogo{socg-logo.pdf}
\SeriesVolume{367}
\ArticleNo{XX}     

\renewcommand{\epsilon}{\varepsilon}
\renewcommand{\phi}{\varphi}

\DeclareMathOperator{\tw}{tw}

\newcommand{\calP}{\ensuremath{\mathcal{P}}\xspace}
\newcommand{\calG}{\ensuremath{\mathcal{G}}\xspace}
\newcommand{\Hyp}{\ensuremath{\mathbb{H}}\xspace}

\DeclareMathOperator{\rad}{\rho}

\begin{document}

\maketitle

\begin{abstract}
    Hyperbolic uniform disk graphs (HUDGs) are intersection graphs of disks with some radius $r$ in the hyperbolic plane, where $r$ may be constant or depend on the number of vertices in a family of HUDGs.
    We show that HUDGs with constant clique number do not admit \emph{product structure}, i.e., that there is no constant $c$ such that every such graph is a subgraph of $H \boxtimes P$ for some graph $H$ of treewidth at most $c$.
    This justifies that HUDGs are described as not having a grid-like structure in the literature, and is in contrast to unit disk graphs in the Euclidean plane, whose grid-like structure is evident from the fact that they are subgraphs of the strong product of two paths and a clique of constant size [Dvořák et al., '21, MATRIX Annals].
    By allowing $H$ to be any graph of constant treewidth instead of a path-like graph, we reject the possibility of a grid-like structure not merely by the maximum degree (which is unbounded for HUDGs) but due to their global structure.
    We complement this by showing that for every (sub-)constant $r$, HUDGs admit product structure, whereas the typical hyperbolic behavior is observed if $r$ grows with the number of vertices. 
    
    Our proof involves a family of $n$-vertex HUDGs with radius $\log n$ that has bounded clique number but unbounded treewidth, and one for which the ratio of treewidth and clique number is $\log n / \log \log n$.
    Up to a $ \log \log n $ factor, this negatively answers a question raised by Bläsius et al. [SoCG '25] asking whether balanced separators of HUDGs with radius $\log n$ can be covered by less than $ \log n $ cliques.
    Our results also imply that the local and layered tree-independence number of HUDGs are both unbounded, answering an open question of Dallard et al. [arXiv '25].
\end{abstract}

\clearpage

\section{Introduction}
\label{sec:intro}

In this paper, we investigate the structure of \emph{hyperbolic uniform disk graphs (HUDGs)}, i.e., graphs $ G $ for which there is an $ r $ such that $ G $ is the intersection graph of disks with radius $ r $ in the hyperbolic plane.
The structure of these graphs heavily depends on $ r $, where $ r $ may be constant or a function depending on the number $ n $ of vertices for a family of HUDGs.
In the literature~\cite{blasius_strongly_2023, Struc_Indep_Hyper_Unifor_Disk_Graph-Blaesius25}, HUDGs with very small $ r $ are described as almost Euclidean or grid-like, and as firmly hyperbolic or more hierarchical for larger $ r $; see Figure~\ref{fig:intro-hudg-radius}.
However, the state of the art cannot fully explain in what sense the grid-like structure depends on $ r $.
We contribute to the understanding of the structure of HUDGs by formalizing when a graph has a grid-like structure and showing that this is indeed not satisfied by HUDGs, unless the radius is small.
We do so by showing that HUDGs do not admit a so-called \emph{product structure}, and thereby also deepen the understanding of this notion.
Before stating our results precisely, we review what is already known on the structure of HUDGs and how it depends on the radius $r$.

\begin{figure}[b]
    \centering
    \includegraphics[scale=0.85]{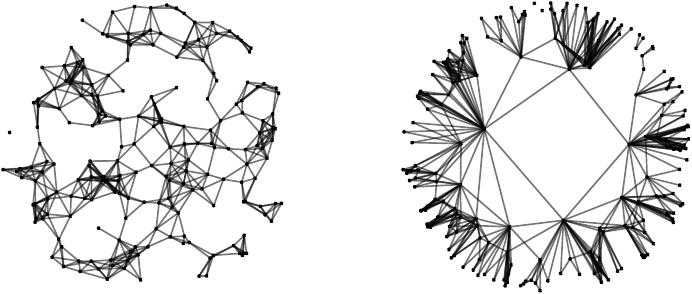}
    \caption{HUDGs with random vertex positions.
      Left: The radius $r$ is so small (think of $1/\sqrt{n}$) that it is indistinguishable from the Euclidean setting.
      Right: The disk radius $r$ is logarithmic in $n$.}
    \label{fig:intro-hudg-radius}
\end{figure}

We start with small radii and work our way up from there.
Locally, the hyperbolic plane behaves like the Euclidean plane, which sets the expectation that very small $r$ yields graphs similar to \emph{Euclidean unit disk graphs (EUDGs)}, which are defined as intersection graphs of unit disks in the Euclidean plane.
While there are always HUDGs that are not EUDGs independent of how small we choose $r$, it is true that every EUDG is in fact a HUDG if we choose $r$ sufficiently small~\cite{blasius_strongly_2023}.
Sufficiently small here means that $r$ has to shrink with the number of vertices $n$.
EUDGs~\cite{clark_unit_1990} are a well studied graph class that allows for balanced separators that can be covered with $ O(\sqrt{n}) $ cliques~\cite{Framew_Expon_Time_Hypot-Berg20}.
We remark that balanced separators are arguably one of the main reasons why HUDGs, EUDGs, and more generally disk graphs, are interesting from an algorithmic perspective.
For HUDGs, there are balanced separators that can be covered with $ O((1 + 1/r) \log n) $ cliques~\cite{Struc_Indep_Hyper_Unifor_Disk_Graph-Blaesius25}.
Interestingly, this matches the $\sqrt{n}$ bound of the Euclidean case for $r \in \Theta(1/\sqrt{n})$ up to a $\log n$ factor and improves for larger $r$.
In fact, if $r \in \Omega(1)$, this gives separators coverable with $\log n$ cliques.
This gives a strong argument that the structure of HUDGs changes when going from sub-constant to at least constant $r$.
This change in structure is also reflected in the fact that several problems can be solved more efficiently for constant $r$ than in the Euclidean setting~\cite{kisfaludi_hyperbolic_2019,Struc_Indep_Hyper_Unifor_Disk_Graph-Blaesius25}.
Summarizing, there is strong evidence that the structure of HUDGs change significantly between sub-constant $r$ and constant $r$.

This raises the question whether there is a similar structural difference between constant $r$ and super-constant $r$.
While the above separators do not continue to get smaller, there are some indications that super-constant $r$ makes a difference, and that in particular $r \in \Theta(\log n)$ is an interesting case.
For HUDGs with $r \in \Theta(\log n)$, the independent set problem can be solved in polynomial time~\cite{Struc_Indep_Hyper_Unifor_Disk_Graph-Blaesius25}, which is in contrast to constant $r$~\cite{kisfaludi_hyperbolic_2019}.
Moreover, HUDGs with radius $r \in \Theta(\log n)$ and random vertex positions are so-called hyperbolic random graphs~\cite{krioukov_hyperbolic_2010}; a random graph model that is popular due to its resemblance to real-world networks in regards to properties like clustering, diameter, and degree distribution~\cite{friedrich_diameter_2018,friedrich_diameter_2018_conf,Random_Hyper_Graph-Gugel12}.
While this indicates a significant difference between constant and growing $r$, the state of the art does not make the structural differences explicit.
The only structural insight we are aware of is that large stars, or more generally high-degree vertices with sparse neighborhoods, require $r$ to grow logarithmically with the degree~\cite{blasius_strongly_2023}.

The main result of this paper is to show that every class of HUDGs with constant radius $r$ has product structure, while there are HUDGs with growing radius that do not.
This closes two gaps in the literature: First, it justifies why HUDGs with large radii should be considered \emph{not} grid-like.
And second, the contrast to HUDGs with constant radius shows a stark structural difference between constant radius compared to larger radii, say $ \log n $.

\begin{figure}
    \centering
    \includegraphics{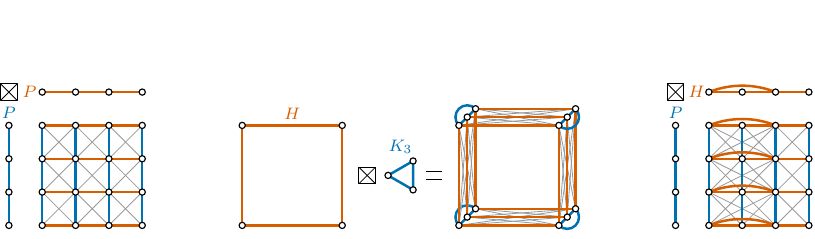}
    \caption{The strong product of two paths (left), of a graph $ H $ and a clique (middle), and a graph $ H $ with a path (right)}
    \label{fig:intro_product_structure}
\end{figure}

Concerning the notion of grid-like structure, observe that $\sqrt{n} \times \sqrt{n}$-grids are EUDGs and thus also HUDGs for small radius.
However, the grid-like structure of EUDGs goes far beyond merely containing grids.
In fact, EUDGs are in general rather grid-like in the sense that there may be large cliques locally, but the global structure of every EUDG behaves like (the subgraph of) a grid.
This notion of being \enquote{grid-like} can be formalized using the \emph{strong product} of graphs, which we introduce formally in \cref{sec:preliminaries}.
For now, it suffices to know that the strong product $P_a \boxtimes P_b$ of two paths on $a$ and $b$ vertices, respectively, is the $a \times b$-grid with diagonals.
Moreover, the strong product $H \boxtimes K_k$ of a graph $H$ with the complete graph $K_k$ on $k$ vertices expands each vertex of $H$ to a $k$-clique as shown in \cref{fig:intro_product_structure}.
Dvořák, Huynh, Joret, Liu, and Wood~\cite{dvorak_notes_2021} show that every EUDG $G$ is a subgraph of $P \boxtimes P \boxtimes K_k$, where $P$ is a path and $k$ is linear in the clique number $ \omega $ of $G$.
Equivalently, one can partition the vertices of a EUDG into sets of size $O(\omega)$ such that contracting each subset into a single vertex yields a subgraph $ H $ of a grid (with diagonals), yielding a product of the form $ H \boxtimes K_k $.

More generally, the concept of product structure allows us to view the graphs of a graph class as subgraphs of the product of simpler graphs; usually a graph with small treewidth and a path.
We say a graph class \calG admits \emph{product structure} if for every $ G \in \calG $, we have $ G \subseteq H'' \boxtimes P $ for some graph $ H'' $ of constant treewidth and some path $ P $.
In our proofs, we show a stronger statement, namely that the graphs we consider are subgraphs of $ H \boxtimes K_k $ for some hyperbolic tiling $ H $.
Since planar graphs, which include hyperbolic tilings, admit product structure, this implies $ H \subseteq H' \boxtimes P \boxtimes K_3 $ for some graph $ H $ of constant treewidth~\cite{dujmovic_planar_2019,dujmovic_planar_2019_conf}, so with $ H'' = H' \boxtimes K_3 \boxtimes K_k $ we obtain a product of the form $ H'' \boxtimes P $.
Note that the above mentioned result on EUDGs can be viewed in this framework by choosing $ H = P \boxtimes P $.
%
%
We remark that replacing one of the paths in the product $ P \boxtimes P \boxtimes K_k $ by a graph $ H' $, enables graphs that go beyond Euclidean square-grids.
In particular, it allows for high-degree vertices.
That is, rejecting a grid-like structure via product structure is significantly stronger than showing that a graph class does not allow for a decomposition of the form $ P \boxtimes P \boxtimes K_k $.
To finish the discussion on grid-like structure, let us mention that \enquote{having product structure} does not need to be seen as a binary decision.
For this, the \emph{row-treewidth} of a graph $ G $ is the smallest $ t $ such that $ G \subseteq H'' \boxtimes P $ for some graph $ H'' $ of treewidth $ t $ and some path $ P $.
The row-treewidth of a graph class is the maximum row-treewidth among all contained graphs, possibly as a function depending on the number of vertices if it is unbounded.
Now, among graph classes admitting product structure, we consider a class to have a stronger grid-like structure if the row-treewidth is small.
That is, in the strongest case of row-treewidth 1, $ H $ is a tree or even a path.
Similarly, a graph class is close to having product structure if the function bounding the row-treewidth is small.

Our results that some families of HUDGs admit product structure while others do not, 
make HUDGs also interesting from the product structure perspective.
Product structure has proved to be a useful tool for many graph-theoretic (see \cite{hickingbotham_shallow_2024} for a survey) and algorithmic~\cite{an_PS_algo_application_2025,bose_optimal_2022} applications.
While naturally, the literature concentrates on graph classes that admit product structure~\cite{dujmovic_planar_2019,campbell_product_2022,hickingbotham_shallow_2024,hickingbotham_product_2024,distel_improved_2022,dujmovic_graph_2023}, it is also interesting to find graph classes that do not admit product structure, for an other reason than the presence of large cliques.
Interestingly, many known constructions are geometric intersection graphs, namely different kinds of string graphs with constant clique number~\cite{merker_intersection_2024,karol_string_product_structure_2025}.
\todo{add non-geometry graph classes? \cite{bose_ltw_rtw_2022}, apex-grids (e.g. $ K_6 $-minor free graphs), 3-regular graphs, 1-gap planar graphs, and 3-quasi planar graphs}
By showing that HUDGs do not admit product structure, we contribute to an improved understanding of the limitations of this notion.

\subsection{Main Results and Discussion}

The following theorem states our main result that HUDGs do not admit product structure, even when restricted to constant clique size, rejecting the possibility of a global grid-structure. 
\begin{restatable}{theorem}{mainNoPSApp}\label{main:noPS}
    Hyperbolic uniform disk graphs with constant clique number do not admit product structure.
\end{restatable}
Our proof of \cref{main:noPS} requires the radius to grow with the graph size.
This is indeed necessary as HUDGs with constant or sub-constant disk radius do admit product structure.
\begin{restatable}{theorem}{mainPSApp}\label{main:rPS}
    Every family of hyperbolic uniform disk graphs with clique number and disk radius in $ O(1) $ admits product structure.
    In contrast, for every super-constant $ r $, there are families of hyperbolic uniform disk graphs with constant clique number and disk radius in  $ \Theta(r) $ not admitting product structure.
\end{restatable}%
This confirms the observations of previous results~\cite{Struc_Indep_Hyper_Unifor_Disk_Graph-Blaesius25,blasius_strongly_2023} that HUDGs become similar to the Euclidean case for small radii, whereas they differ significantly for large radii.
As having product structure in some sense restricts the complexity of the neighborhood of an individual vertex, this in particular indicates that HUDGs become more complex (and for some applications more interesting) when allowing the radius to grow with the graph size.
We show the first part of \cref{main:rPS} in \cref{sec:upper_bound} and the second in \cref{sec:lower_bound}.

\subparagraph{Bounds on the Row-Treewidth.}

We actually prove more fine-grained results that go beyond the binary question of having product structure.
Recall that a graph class admits product structure if and only if it has bounded row-treewidth, i.e., a super-constant lower bound and a constant upper bound on the row-treewidth disprove and prove product structure, respectively.
Our lower and upper bounds on the row-treewidth depending on clique number $\omega$ and disk radius $r$ are shown in \cref{tab:bounds-on-rtw}.
The first two rows show our lower bounds that grow with $ r $ (as long as $ r \in O(\log n)$), where we have the two cases $ \omega \in O(1) $ and $ \omega \in O(\log \log n) $.
The first case is most relevant for product structure as growing clique size immediately rejects product structure, whereas the second case gives the stronger lower bound.
Then, the third row gives our upper bound,
whereas the last row is from the literature and is discussed at the end of this section.
Before doing so, we discuss each result in \cref{tab:bounds-on-rtw} in more detail.

\begin{table}[htbp]
    \centering
    \caption{Overview of lower and upper bounds for the row-treewidth.
      The lower bounds mean that for every radius $r$ in the given regime, there exists a family of HUDGs with radius $\Theta(r)$ and the given bounds on the clique number $\omega$ and the row-treewidth.
      The upper bounds hold for all HUDGs with radius $r$ and clique number $\omega$ (both of which may depend on the graph size $n$).
    }
    \label{tab:bounds-on-rtw} \begin{tabular}{c>{\centering\arraybackslash}p{1.8cm}>{\centering\arraybackslash}p{1.8cm}l}
      \toprule
      $\boldsymbol{\omega}$ & \multicolumn{2}{c}{\textbf{row-treewidth}}                              & \textbf{comment} \\
                            & \footnotesize $r \in O(\log n)$  & \footnotesize $r \in \Omega(\log n)$ & \\
      \toprule
      $O(1)$      & $\Omega(\log r)$ & $\Omega(\log\log n)$ & Cor.~\labelcref{cor:lower-bound-sub-log-radius,cor:lower-bound-super-log-radius}; implies Thm.~\labelcref{main:noPS,main:rPS}\\
      $O(\log\log n)$ & $\Omega(r)$      & $\Omega(\log n)$     & Cor.~\labelcref{cor:lower-bound-sub-log-radius,cor:lower-bound-super-log-radius}          \\
      \midrule
      variable & \multicolumn{2}{c}{$O(\omega \cdot 3^{8r})$}             & Thm.~\labelcref{thm:upper_bound}; implies Thm.~\labelcref{main:rPS} for $r, \omega \in O(1)$        \\
      variable & \multicolumn{2}{c}{$O(\omega (1 + \frac{1}{r}) \log n)$} & follows from treewidth bound \cite{Struc_Indep_Hyper_Unifor_Disk_Graph-Blaesius25} \\
      \bottomrule
    \end{tabular} 
\end{table}

The first lower bound in \cref{tab:bounds-on-rtw} assumes constant clique number $\omega \in O(1)$.
Note that there are two regimes for the radius $r$.
For radii smaller than $\log n$, the lower bound grows with increasing radius until it reaches a row-treewidth of $\Omega(\log\log n)$ for radius $r \in \Theta(\log n)$.
For larger radii, the lower bound stops to grow with $ r $.
We note that any super-constant radius already implies super-constant row-treewidth and thus refutes product structure, yielding \cref{main:noPS} and the second part of \cref{main:rPS}.
That is, the specific lower bounds are significantly stronger than \cref{main:noPS} as they refute product structure for every growing radius and not only for the class of all HUDGs.
For the lower bound given in the next row of \cref{tab:bounds-on-rtw}, we allow slightly growing clique number $\omega \in O(\log \log n)$, which yields exponentially larger lower bounds for the row-treewidth compared to the first row.
We remark that all our lower bounds are linear in the size of the largest clique-minor, also known as the Hadwiger number.

Turning to our upper bound (bottom part of \cref{tab:bounds-on-rtw}), observe that it only depends on $\omega$ and $r$.
Assuming both to be constant directly implies the first part of \cref{main:rPS}.
Moreover, for growing clique number $\omega$ (but constant $r$), this is the best upper bound one can hope for, as the clique number is a lower bound on the row-treewidth.
Concerning growing radius, note that the radius $r$ contributes exponentially to the row-treewidth.
While it might be possible to improve this upper bound, we note that our lower bounds for growing radius imply that there has to be a dependence on $r$.
We remark that we strengthen our upper bound in terms of product structure.
For this, recall that the upper bound on the row-treewidth means that every HUDG with radius $ r $ is a subgraph of $ H'' \boxtimes P $ for some graph $ H'' $ of treewidth $ O(\omega \cdot 3^{8r}) $.
We show the stronger statement that for every HUDG $ G $ with radius $r$ and clique number $\omega$, we have $G \subseteq H \boxtimes K_k \subseteq H' \boxtimes P \boxtimes K_{k}$, for some hyperbolic tiling $ H $, some graph $H'$ of treewidth $3$, some path $P$, and $k \in O(\omega \cdot 3^{8r})$.
Our proof is similar to the Euclidean case~\cite{dvorak_notes_2021}, but instead of the Euclidean square-grid, we construct a suitable tiling of the hyperbolic plane.
Interestingly, the tiling depends on $r$ and becomes more hyperbolic for large $ r $ and more Euclidean for small $ r $.
With our tiling, we also contribute to the rich field of constructing hyperbolic tilings with various favorable properties~\cite{Tesselations_Coxeter97,physics_Thurn25,datta_semi_2019}.
\NewDocumentCommand{\tilingLemma}{m O{$ O(3^{8r})$}}{%
    For each #1 there is an (irregular) tiling of the hyperbolic plane such that 
    \begin{itemize}
        \item every tile can be covered by #2 disks of radius $ r $ and
        \item each two points with distance at most $ 2r $ lie in the same tile or in two adjacent tiles.
    \end{itemize}
    \todo[inline]{Moreover, the dual of the tiling admits a layering $ L_1, L_2, \dots $ such that for each $ i $, the subgraph induced by $ L_i $ and $ L_{i + 1} $ has a matching saturating $ L_i $.}
    \todo{this would reduce the treewidth of $ H $ from 3 to 2 \cite[Thm 5]{hickingbotham_product_2024}}
}
\begin{restatable}{theorem}{tiling}\label{lem:tiling}
    \tilingLemma{$ r > 0 $}

\end{restatable}

The last upper bound in \cref{tab:bounds-on-rtw} directly follows from the literature, which we discuss in the next paragraph.

\subparagraph{Connections to the State of the Art.}

Our lower bounds are based on the fact that the row-treewidth is asymptotically lower-bounded by the treewidth in the neighborhood of a vertex~\cite{dujmovic_layered_2017}.
Notably, the neighborhoods of HUDGs are called \emph{strongly hyperbolic uniform disk graphs (SHUDGs)} and are an interesting graph class in their own right~\cite{blasius_strongly_2023}.
In fact, the much studied hyperbolic random graphs are SHUDGs with random vertex positions~\cite{krioukov_hyperbolic_2010}.
The core of our lower bounds for the row-treewidth of HUDGs are the following bounds on the treewidth of SHUDGs.
\begin{restatable}{theorem}{mainTreewidthApp}\label{main:treewidth}
    There are families of $ n $-vertex strongly hyperbolic uniform disk graphs with radius $\Theta(\log n)$ and
    \begin{itemize}
        \item clique number $ O(\log \log n) $ and treewidth $ \Omega(\log n) $
        \item clique number $ O(1) $ and treewidth $ \Omega(\log \log n) $.
    \end{itemize}
    
\end{restatable}
To connect this to the results by Bläsius, von der Heydt, Kisfaludi-Bak, Wilhelm, and Van Wordragen~\cite{Struc_Indep_Hyper_Unifor_Disk_Graph-Blaesius25}, recall that they show that $ n $-vertex HUDGs with radius $ r $ have balanced separators that can be covered with $ (1 + 1/r) \log n $ cliques.
This implies that the treewidth is $ O(\omega (1 + 1/r) \log n) $ \cite{Treew_Graph_Balan_Separ-DvoraNorin19}, which is an upper bound for the row-treewidth (which is why we list it in \cref{tab:bounds-on-rtw}).
For $r$ at least constant, this gives separators coverable with $O(\log n)$ cliques, i.e., treewidth $ O(\omega \log n) $. 
The authors of \cite{Struc_Indep_Hyper_Unifor_Disk_Graph-Blaesius25} raise the question whether this can be improved, in particular for increasing radius $r$.
With \cref{main:treewidth}, we give a negative answer to this, up to a $\log\log n$ factor: The family with clique number $\omega \in O(\log\log n)$ has treewidth $\Omega(\log n) = \Omega(\omega \log n / \log\log n)$.

In addition, \cref{main:treewidth} provides the first known example of a family of SHUDGs with constant clique number but unbounded treewidth.
Graph classes with this property are also called \emph{not $ (\tw, \omega) $-bounded}.
By Dallard, Milanič, and Štorgel~\cite{dallard_tree-alpha_2024}, graph classes with bounded tree-independence number are $ (\tw, \omega) $-bounded, so our constructed family of SHUDGs has unbounded tree-independence number.
Since SHUDGs are neighborhoods of vertices in HUDGs, it follows that the local tree-independence number of HUDGs is unbounded, which in turn implies that their layered tree-independence number is unbounded~\cite{galby_local-tree-alpha_2024}.
This gives a negative answer to an open question of Dallard, Milanič, Munaro, and Yang~\cite{dallard_layered-tree-alpha_2025}.

\begin{corollary}\label{cor:tree-alpha}
    The local tree-independence number and the layered tree-independence number of hyperbolic uniform disk graphs are unbounded.
\end{corollary}

Note that for these two answered open questions, the first comes from the HUDG side and the other is more related to product structure.
This underscores that HUDGs and product structure form an interesting combination.

\todo{I really dislike to have the last sentence at this position. Any better suggestions?}
Also note that \cref{cor:tree-alpha} shows another structural difference between HUDGs with super-constant radius and those with constant radius, as the layered tree-independence number is bounded for the latter~\cite{dallard_layered-tree-alpha_2025}.



\section{Preliminaries}\label{sec:preliminaries}

\subsection{Basic Graph Notation}

Let $G = (V, E)$ be a \emph{graph} with \emph{vertex} set $V$ and \emph{edge} set $E \subseteq \binom{V}{2}$.
Unless mentioned otherwise, we consider graphs to be \emph{simple} without \emph{multi-edges} and \emph{self-loops}, i.e., $E$ contains each edge only once and each edge contains two different vertices.
In case we need to distinguish between different graphs, we denote the vertices and edges of $G$ with $V(G)$ and $E(G)$, respectively.
For $uv \in E$, we say $u$ and $v$ are \emph{adjacent} to each other and \emph{incident} to the edge $uv$.
We also call $u$ and $v$ the \emph{endpoints} of $uv$.

Two vertices $u, v \in V$ are \emph{reachable} if there is a sequence $u = v_0, \dots, v_k = v$ of vertices such that consecutive vertices are adjacent, i.e., $v_{i - 1}v_i \in E$ for $i \in [k]$.
The minimum $k$ for which such a sequence exists is the \emph{distance} between $u$ and $v$ in $G$.
Reachability is an equivalence relation between vertices and the equivalence classes are the \emph{connected components}.
A graph is \emph{connected} if it has only one connected component.


The \emph{neighborhood} $N(v)$ of a vertex $v \in V$ is the set of vertices adjacent to $v$, i.e., $N(v) = \{u \in V \mid uv \in E\}$.
The \emph{closed neighborhood} $N[v] = N(v) \cup \{v\}$ additionally contains $v$ itself.
The \emph{degree} of $v$ is the size $|N(v)|$ of its neighborhood.
A vertex is \emph{isolated} if it has degree~$0$.  

\subparagraph{Graph Operations.}

We denote the graph obtained from $G = (V, E)$ by deleting a single vertex $v \in V$ or a subset of vertices $S \subseteq V$ with $G - v$ and $G - S$, respectively.
For a vertex subset $S \subseteq V$, \emph{contracting} $S$ is the operation of replacing all vertices in $S$ and their occurrences in edges by just a single vertex; removing multi-edges and self-loops afterwards.
A graph $H$ is a \emph{subgraph} of $G$ if $V(H) \subseteq V(G)$ and $E(H) \subseteq E(G)$, i.e., $H$ can be obtained from $G$ by deleting edges and vertices.
It is an \emph{induced} subgraph if it can be obtained by just vertex deletions.
In this case, $H$ is completely determined by $G$ and $V(H)$ and we denote the induced subgraph with $G[V(H)]$.
The graph $H$ is a \emph{minor} of $G$ if $H$ can be obtained from $G$ by taking a subgraph and contracting edges.
For a partition $\mathcal P$ of the vertices $V$, the \emph{quotient} $G / \mathcal P$ is the graph obtained from $G$ by contracting each block of $\mathcal P$.

\subparagraph{Specific Graphs.}

The \emph{complete graph} $K_n$ with $n = |V|$ vertices is the graph where every pair of vertices is adjacent, i.e., $E = \binom{V}{2}$.
A complete graph is also called \emph{clique} (usually, when it appears as subgraph in a different graph).
The \emph{clique number} $\omega$ of a graph $G$ is the largest clique that is a subgraph of $G$.

A \emph{path} is a connected graph in which every vertex has degree~$2$ except for two vertices with degree~$1$.
The two degree-$1$ vertices are also called the \emph{endpoints} of the path.
We usually denote paths with $P$.
A \emph{cycle} is a connected graph in which every vertex has degree~$2$.
A cycle with $3$ vertices is also called \emph{triangle}.
A graph containing no cycle as subgraph is called \emph{acyclic}.
A \emph{tree} is a graph that is acyclic and connected.
A \emph{spanning tree} of a graph $G$ is a subgraph of $G$ that is a tree on all vertices of $G$.
A \emph{rooted tree} is a tree together with a designated vertex called \emph{root}.
Let $T = (V, E)$ be a tree with root $v \in V$.
For every non-root vertex $u \in V$, $T$ contains a unique path from $u$ to the root $v$.
The neighbor of $u$ on this path is the \emph{parent} of $u$.
Moreover, $u$ is a \emph{child} of its parent.
The rooted tree $T$ is \emph{binary} if every vertex has two children.
The \emph{ancestor} and \emph{descendant} relations are the transitive closures of the parent and child relations, respectively.

\subsection{Product Structure}

\subparagraph{Separators and Treewidth.}

Let $G = (V, E)$ be a connected graph with $n = |V|$ vertices.
A vertex set $S \subseteq V$ is a \emph{separator} if deleting $S$ disconnects $G$.
The separator is \emph{balanced} if each connected component in $G - S$ has at most $2/3 \cdot n$ vertices.
We note that every tree has a balanced separator of size $|S| = 1$.
The \emph{treewidth}\footnote{For this paper, we do not require a formal definition of treewidth via tree decompositions. See \cite[Section~7]{cygan_treewidth_2015} for a thorough definition.} $\tw(G)$ of a graph $G$ is a parameter that measures how similar $G$ is to a tree with respect to balanced separators.
Essentially, $\tw(G)$ is the size of such a balanced separator and trees have treewidth~$1$.
This can be made more precise as follows.
Dvořák and Norin~\cite{Treew_Graph_Balan_Separ-DvoraNorin19} showed that if every subgraph of $G$ has a balanced separator of size at most $s$, then $\tw(G) \le 15s$.
Conversely, it is well known that $G$ has a balanced separator of size at most $\tw(G)$~\cite[Lemma~7.20]{cygan_treewidth_2015}.
The latter also holds for every subgraph of $G$ as the treewidth \emph{minor closed}, i.e., if $H$ is a minor of $G$ then $\tw(H) \le \tw(G)$.
Thus, if $s$ is the smallest number such that every subgraph of $G$ has a balanced separator of size $s$, then the treewidth of $G$ is $s$, up constant factors.
The complete graph $K_n$ on $n$ vertices has treewidth $\tw(K_n) = n - 1$.
Thus, if $G$ has a clique-minor of size $k$, then $\tw(G) \ge k - 1$.

\subparagraph{Strong Product.}

For two graphs $G$ and $H$, the \emph{strong product} $G \boxtimes H$ is defined as follows; also see Figure~\ref{fig:prelim-prod-scruct}.
$ G \boxtimes H $ has vertex set $V(G \boxtimes H) = V(G) \times V(H)$, i.e., every vertex $u = (u_G, u_H)$ of $G \boxtimes H$ corresponds to a vertex $u_G$ in $G$ and to a vertex $u_H$ in $H$.
Two vertices $u = (u_G, u_H)$ and $v = (v_G, v_H)$ are adjacent in $G \boxtimes H$ if one of the following is true.
\begin{itemize}
  \item $u$ and $v$ are the same vertex in $G$ and adjacent in $H$, i.e., $u_G = v_G$ and $u_Hv_H \in E(H)$.
  \item $u$ and $v$ are the same vertex in $H$ and adjacent in $G$, i.e., $u_H = v_H$ and $u_Gv_G \in E(G)$.
  \item $u$ and $v$ are adjacent in $G$ and $H$, i.e., $u_Gv_G \in E(G)$ and $u_Hv_H \in E(H)$.
\end{itemize}
\begin{figure}
    \centering
    \includegraphics{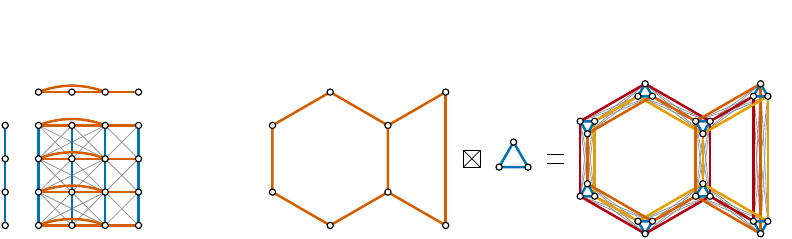}
    \caption{Left: Strong product of a graph $H$ (red) with a path $P$ (blue).
      The three different types of edges are color coded in red (edge in $H$, same vertex in $P$), blue (edge in $P$, same vertex in $H$), and gray (edge in $H$ and $P$).
      Right: Strong product of a graph (red) with a $K_3$ (blue).}
    \label{fig:prelim-prod-scruct}
\end{figure}

\subparagraph{Row-Treewidth and Product Structure.}

We are usually interested in products of the form $H \boxtimes P$ where $H$ has small treewidth and $P$ is a path (of arbitrary length).
Assume $G$ is a subgraph of such a product $G \subseteq H \boxtimes P$ and assume that $H$ is chosen such that $\tw(H)$ is minimum with this property.
Then, we call $\tw(H)$ the \emph{row-treewidth} of $G$.
We say that a graph family $\calG$ has \emph{product structure} if the graphs in $\calG$ have row-treewidth $O(1)$, i.e., there exists a constant that upper-bounds the row-treewidth of every graph $G \in \calG$.

For graph products of the form $H \boxtimes P$, it is often helpful to think of $H \boxtimes P$ as a chain of copies of $H$.
More precisely, if $P$ has $k$ vertices, then $H \boxtimes P$ contains $k$ copies $H_1, \dots, H_k$ of $H$.
Note that with this perspective, a copy $v_i \in V(H_i)$ of $v \in V(H)$ is adjacent to its neighboring copies $v_{i - 1} \in V(H_{i - 1})$ and $v_{i + 1} \in V(H_{i + 1})$, as well as to the copies of $v$'s neighbors in $H_{i - 1}$, $H_{i}$, and $H_{i + 1}$.
Similarly, when considering the product $G = H \boxtimes K_k$, then $G$ can be obtained from $H$ by replacing each vertex in $H$ with a $k$-clique and replacing each edge in $H$ with a complete bipartite graph between the corresponding $k$-cliques.
The treewidth of $H \boxtimes K_k$ is $ (\tw(H) + 1) \cdot k - 1$.

\subsection{Hyperbolic Geometry and Tilings}

While this paper is about graphs in the hyperbolic plane, most of it can be understood with only little knowledge of hyperbolic geometry.
We thus only provide a brief introduction.
Formally, hyperbolic geometry is obtained from Euclidean geometry by replacing the parallel axiom: Instead of having exactly one parallel to a line through some given point not on the line, there are infinitely many, where two lines are considered parallel if they do not intersect.
Thus, many concepts familiar from Euclidean geometry still apply as long as the parallel axiom is not involved, e.g., definitions of distances, angles, disks, (regular) polygons, and congruency are unchanged and the triangle inequality holds.
However, some geometric laws like trigonometric formulas work differently,
which we only introduce when needed.
They can be found in textbooks on hyperbolic geometry, e.g., \cite{Euclid_Non_Euclid_Geomet-Green93}.

\subparagraph{Polar Coordinates.}

We often identify points in the hyperbolic plane using polar coordinates.
For an arbitrarily chosen origin $O$ and reference ray starting in $O$, a point $(\rad, \varphi)$ is defined by its distance $\rad$ to $O$ called \emph{radius} and its \emph{angle} $\varphi$ to the reference ray.
The hyperbolic distance between two points $p_1 = (\rad_1, \varphi_1)$ and $p_2 = (\rad_2, \varphi_2)$ is determined by the two radii $\rad_1$, $\rad_2$, and the \emph{angular distance} $\min\{|\phi_1 - \phi_2|, 2 \pi - |\phi_1 - \phi_2|\}$.
For a fixed hyperbolic distance $d$ and given radii $\rad_1$ and $\rad_2$ for the points $p_1$ and $p_2$, respectively, we denote with $\theta_d(\rad_1, \rad_2)$ the angular distance between $p_1$ and $p_2$ such that they have distance $d$.
While we do not give a formula for the distance between points given in polar coordinates (see, e.g., \cite{blasius_strongly_2023}), we need a formula for $\theta_d(r_1, r_2)$ , which is
\begin{equation}
    \label{eq:max_angular_distance}
    \theta_d(r_1,r_2) = \arccos\left(\frac{\cosh(r_1)\cosh(r_2)-\cosh(d)}{\sinh(r_1)\sinh(r_2)}\right),
\end{equation}
with $\sinh(x) = (e^x - e^{-x}) / 2$ and $\cosh(x) = (e^x + e^{-x}) / 2$.

\subparagraph{Basic Properties.}

Our illustrations sometimes use the Poincaré disk model, which maps the hyperbolic plane into a unit disk of the Euclidean plane.
However, we do not need any knowledge about this model and the illustrations should be helpful regardless.
The core property of the hyperbolic plane relevant for this paper is the following.
The hyperbolic plane expands exponentially, while at the same time behaving like the Euclidean plane locally.
To make this more concrete, consider a disk of radius $r$.
If the radius $r$ is small, then the area and circumference of the disk are quadratic and linear in $r$, respectively, as we are used to in the Euclidean plane.
However, if the radius $r$ increases, the area and circumference both grow exponentially like $\Theta(e^r)$.
This has two interesting effects.
First, compared to the Euclidean plane, the hyperbolic plane has more space.
This effect becomes more pronounced at bigger scales, i.e., when we look at a large region.
Secondly, in Euclidean geometry one can scale objects to change their size while keeping them structurally the same, e.g., think of similar triangles.
In the hyperbolic plane, however, the scale makes a big structural difference, and there is no scaling operation that changes all distances by the same factor.

\begin{figure}
    \centering
    \includegraphics{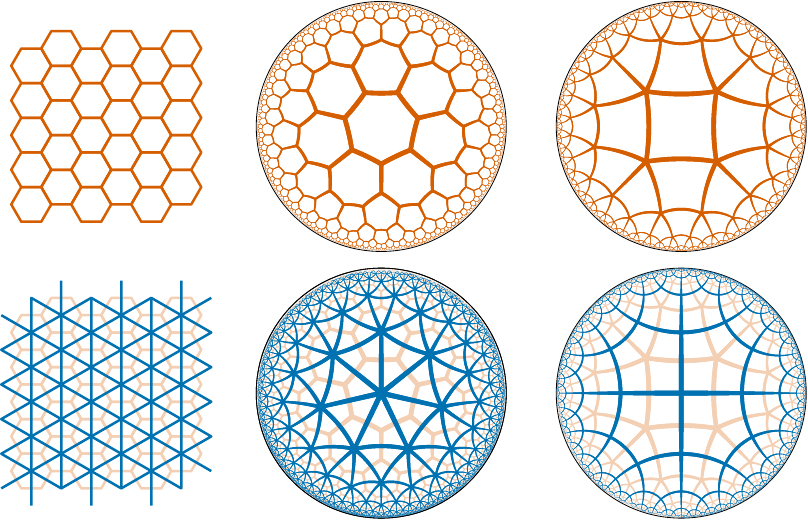}
    \caption{Regular tilings (top) together with their duals (bottom).
      From left to right: A Euclidean $\{6, 3\}$ tiling, a hyperbolic $\{7, 3\}$ tiling, and a hyperbolic $\{4, 5\}$-tiling.
      The hyperbolic tilings are shown in the Poincaré disk.
      While the tiles seem to get smaller towards the boundary of the disk, they are all congruent.}
    \label{fig:prelim-tiling}
\end{figure}

\subparagraph{Hyperbolic Tilings.}

A \emph{regular tiling} is the partition of the plane into congruent regular polygons; also see \cref{fig:prelim-tiling}.
The Euclidean plane can be tiled with triangles, squares, and hexagons.
These tilings have the Schläfli symbols $\{3, 6\}$, $\{4, 4\}$, and $\{6, 3\}$, respectively, where $\{p, q\}$ indicates that each tile is a $p$-gon and $q$ such $p$-gons meet at each corner.
For any $p$ and $q$ with $\frac{1}{p} + \frac{1}{q} < \frac{1}{2}$, we obtain a regular tiling of the hyperbolic plane; see \cref{fig:prelim-tiling} for examples.
While Euclidean tilings can be scaled arbitrarily, the size of the tiles in a hyperbolic $\{p, q\}$-tiling is determined by $p$ and $q$.
The following lemma gives formulas for the size of tiles and directly follows from basic trigonometry.
\begin{lemma}
    \label{lem:distances-in-tilings}
    Consider a tile of a hyperbolic $\{p, q\}$-tiling and let $\ell$ be its side length and $r_1$ and $r_2$ be the radius of the incircle and circumcircle, respectively.
    Then the following holds.
    \begin{equation*}
        \cosh \frac{\ell}{2} = \frac{\cos \frac{\pi}{p}}{\sin \frac{\pi}{q}}
        \hfill
        \cosh r_1 = \frac{\cos \frac{\pi}{q}}{\sin \frac{\pi}{p}}
        \hfill
        \cosh r_2 = \cot \frac{\pi}{p} \cdot \cot \frac{\pi}{q}
        \hfill
    \end{equation*}
    For the $\{4, 5\}$-tiling, we get $\ell / 2 \approx 0.62$ and $r_1 \approx 0.53$.
    For the $\{7, 3\}$-tiling, we get $r_2 \approx 0.62$.
\end{lemma}
Also observe in \cref{fig:prelim-tiling} that the exponential expansion of the hyperbolic plane is reflected in the tilings, which can be seen as a discretization of the plane.
We also consider (irregular) tilings, where the tiles do not need to be congruent.

We usually interpret tilings (regular and irregular) as infinite\footnote{Throughout the paper, we implicitly assume that all considered graphs are finite.
  Tilings are the only exception.
  But also for tilings, we are usually only interested in a finite subgraph.} \emph{plane graphs}, i.e., graphs that are drawn without edge crossings.
Then the corners are the vertices, the sides are the edges, and the $p$-gons are the \emph{faces}.
The \emph{dual} of such a plane graph has one vertex per face and two faces are adjacent if they share an edge.
Again, see \cref{fig:prelim-tiling} for an illustration and note that the $\{p, q\}$-tiling is dual to the $\{q, p\}$-tiling.

\subsection{Geometric Graphs}

A graph $G$ is a \emph{(Euclidean) unit disk graph (EUDG)} if it is the intersection graph of disks of radius $1$ in the Euclidean plane, i.e., each vertex is represented by a disk and two vertices are adjacent if and only if the corresponding disks intersect.
We also interpret the disk centers as the \emph{positions} of the corresponding vertices and use terms like \emph{distance}, i.e., two vertices are connected if and only if their distance is at most $2$.
Analogously, one can define intersection graphs of equally sized disks in the hyperbolic plane.
We say that $G$ is a \emph{hyperbolic uniform disk graph (HUDG)} if there exists a radius $r$ such that $G$ is the intersection graph of hyperbolic disks of radius $r$.
We note that the radius $r$ can depend on the graph and makes an important difference.
To explicitly distinguish the above radius $r$ from the radius of a point in polar coordinates, we often call it \emph{disk radius}.

We are usually interested in the asymptotics of $r$ with respect to the graph size $n$.
We say that a HUDG $G$ \emph{has disk radius $r$} if it is the intersection graph of hyperbolic disks of radius $r$.
Moreover, we say that a family $\calG$ of HUDGs \emph{has disk radius $r(n)$} if every $G \in \calG$ with $n$ vertices has disk radius $r(n)$.
This lets us use asymptotics, i.e., $\calG$ \emph{has disk radius $\Theta(r(n))$} if there exist constants $c_1, c_2$ such that every sufficiently large $n$-vertex graph $G \in \calG$ has disk radius $r_G$ for $c_1 r(n) \le r_G \le c_2 r(n)$.\footnote{We note that this is the natural way to use asymptotics.
  We only make this explicit as asymptotics can be unintuitive in this context, e.g., \enquote{the class of all HUDGs with radius $\Theta(\log n)$} is not well-defined.}
We usually only write $r$ and implicitly assume in the context of a graph class that it is a function depending on $n$ (as we do for other graph parameters like the clique number $\omega$).

As mentioned in the introduction, HUDGs with very small disk radius are almost Euclidean.
To study HUDGs that actually reflect the differences of hyperbolic geometry, one can resort to study families of HUDGs with disk radius in $\Theta(\log n)$~\cite{Struc_Indep_Hyper_Unifor_Disk_Graph-Blaesius25}.
Alternatively, one can look at so called \emph{strongly hyperbolic uniform disk graphs (SHUDGs)}~\cite{blasius_strongly_2023}.
A graph is a SHUDG if it is a HUDG with disk radius $r$ such that all vertices have distance at most $2r$ from the origin.
This means that a vertex placed at the origin would be \emph{universal}, i.e., adjacent to all other vertices of the graph.
Note that the neighborhood of every vertex in a HUDG induces an SHUDG, and conversely every SHUDG is the open or closed neighborhood of some vertex in an (S)HUDG.
Studying SHUDGs is in some sense similar to studying families of HUDGs with large disk radius: SHUDGs with bounded average degree have disk radius $\Omega(\log n)$~\cite{blasius_strongly_2023}.

\section{HUDGs Do Not Admit Product Structure}
\label{sec:large_radii}
\label{sec:lower_bound}

%

In this section, we show that hyperbolic uniform disk graphs with constant clique number do not admit product structure (\cref{main:noPS}), justifying the claim that, in contrast to Euclidean uniform disk graphs, they do not have a grid-like structure.
We remark that while our construction initially has disk radius $ \Theta(\log n) $, 
we extend it to arbitrary super-constant disk radii to prove the second part of \cref{main:rPS}.
We note that this is tight as we show in \cref{sec:upper_bound} that every family of HUDGs with disk radius in $O(1)$ and constant clique number has indeed product structure.
Similarly, we provide families of HUDGs justifying the stronger lower bounds in the second row of \cref{tab:bounds-on-rtw} for every super-constant disk radius.
These results are summarized in \cref{cor:lower-bound-sub-log-radius,cor:lower-bound-super-log-radius} at the end of the section.



To work towards rejecting product structure, first observe that in a grid, the neighborhood of each vertex has bounded size, and in particular bounded treewidth.
Though we explicitly use product structure so our graphs may have arbitrarily large neighborhoods, we exploit that product structure implies the treewidth of the neighborhood of each vertex to be bounded.

\begin{lemma}[{\cite[Lemma 6]{dujmovic_layered_2017}}, {\cite[Lemma 6]{dujmovic_planar_2019}}, see also~\cite{bose_ltw_rtw_2022}]\label{lem:tw_neighborhood}%
\label{lem:PS_neighboorhood}
    For every graph $H$, path $ P $, and every vertex $v$ of $ H \boxtimes P $, the closed neighborhood of $v$ has treewidth at most $ 3 \cdot (\tw(H) + 1) - 1 $.
\end{lemma}
In particular, for every graph class $ \mathcal{G} $ admitting product structure, there is a constant $c$ such that for every graph $G \in \mathcal{G}$ and every vertex $v \in V(G)$, the neighborhood $ N[v] $ has treewidth at most $c$.
In consequence, to reject product structure, it suffices to construct a family of graphs with unbounded treewidth in the neighborhood of some vertex.
Moreover, \cref{lem:tw_neighborhood} shows that, asymptotically, the treewidth in the neighborhood of some vertex is a lower bound on the row-treewidth.
In the specific case of HUDGs, every SHUDG is the open or closed neighborhood of some vertex in a HUDG.
Thus, to prove \cref{main:noPS} we now aim for a family of SHUDGs with constant clique number but unbounded treewidth.

Recall that we restrict ourselves to constant clique number as the treewidth of cliques, and of the neighborhood of its vertices, is linear in its size.
As an intermediate step, however, we start with constructing a family of SHUDGs having clique number $O(\log \log n)$ and treewidth $\Omega(\log n)$, which is a result of independent interest as there already is an exponential gap between the clique number and the treewidth.
As we can assume SHUDGs to have a universal vertex, the same gap holds in the neighborhood of a vertex, yielding an exponential gap between the clique number and the row-treewidth.


\begin{theorem}\label{thm:hypergrid}
    There is a family of $ n $-vertex SHUDGs with disk radius in $\Theta(\log n)$, clique number
    in $O(\log \log n)$, and treewidth in $\Omega(\log n)$.
\end{theorem}

We then identify an induced subgraph within each of the constructed SHUDGs whose clique number is in $\Theta(1)$ and whose
treewidth is in $\Omega(\log \log n)$, proving \cref{main:noPS} not only for HUDGs but even for the strict subclass of the \emph{strongly} hyperbolic uniform disk graphs, and here even for those having disk radius $ \Theta(\log n)$.
In addition, \cref{thm:hypergrid,thm:sub-hypergrid} together prove \cref{main:treewidth}.

\begin{theorem}\label{thm:sub-hypergrid}
    There is a family of $ n $-vertex SHUDGs with disk radius in $\Theta(\log n)$, clique number in $\Theta(1)$,
    and treewidth in $\Omega(\log \log n)$.
\end{theorem}
Comparing the two theorems, we remark that \cref{thm:hypergrid} provides a stronger lower bound on the treewidth, and therefore also row-treewidth, of SHUDGs stating that there is a family of SHUDGs
with treewidth and row-treewidth $\Omega (\omega \frac{\log n}{\log \log n} )$, where $\omega$ denotes the clique number.
Note, however, that we only show that this bound can be obtained for $\omega \in \Theta(\log \log n)$, rather than for an
arbitrary clique number.
Notably, it remains open whether there exists a family of SHUDGs with constant clique number and treewidth $\Omega ( \frac{\log n}{\log \log n} )$.
However, with regard to product structure, we are primarily interested in SHUDGs with constant clique number, whose treewidth and row-treewidth is unbounded by \cref{thm:sub-hypergrid}.


\subsection{Proof of \texorpdfstring{\cref{thm:hypergrid}}{Theorem 8}}
\label{sec:lower_bound_tw_logn}
\todo{adjust theorem number in texorpdfstring}


For each integer $ r \geq 2 $, we construct an $n$-vertex SHUDG $G_r$ by placing $ n \in \Theta(2^r) $ vertices into the hyperbolic plane, which yields $ r \in \Theta(\log n) $ for the disk radius as required by \cref{thm:hypergrid,thm:sub-hypergrid}.
First, we place a vertex at the origin which we call the \emph{root}.
Second, for $ 2 \leq k \leq r $, let $ A_k $ be a regular $ 2^k $-gon with side lengths $ 2r $ whose center is at the origin.
We align these polygons such that for each $ k > 2 $, every other edge of $ A_k $ has the same perpendicular bisector as some edge of $ A_{k-1} $; see \cref{fig:regular_k-gons}.
Now, let $ V_k $ denote the vertex set of $ A_k $ and let $ V_1 $ contain exactly the root.
We say a vertex is in the \emph{$k$-th level} if it is contained in $ V_k $.
The vertex set of $ G_r $ is defined as the union of all $ V_k $, $ k \in [r] $, and two vertices are adjacent if and only if their distance is at most $ 2r $.
In particular, the sides of the polygons are also edges in $ G_r $, and there are no further edges within a level as chords in a regular polygon are strictly longer than their sides.
That is, each level induces a cycle, where the angular distance between any two adjacent vertices in level $ k $ is $ 2 \pi / 2^k $.
Note that the number of vertices is indeed in $ \Theta(2^r) $, as required.
Without going into detail, we note that the above construction is inspired by an analysis of the treewidth of random SHUDGs by Bläsius, Friedrich, and Krohmer~\cite{blasius_hyperbolic_2016}; see \cref{fig:regular_k-gons}.

\begin{figure}
    \centering
    \includegraphics{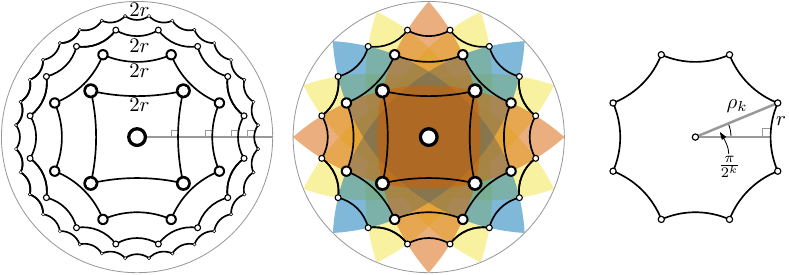}
    \caption{%
      Left: Vertex positions of $ G_r $ constructed for \cref{thm:hypergrid}, where the $ k $-th level, $ k \geq 2 $, induces a $ 2^k $-cycle with consecutive vertices having angular distance $ 2\pi / 2^k $.
      Center: Equivalent construction inspired by the separators in \cite{blasius_hyperbolic_2016}, where the vertices are placed at intersection points of hypercycles.
      Right: A right triangle for $ A_k $ justifying \cref{obs:sinh_rk}.  
    }
    \label{fig:regular_k-gons}
\end{figure}

Before bounding the clique number and treewidth of $ G_r $ for \cref{thm:hypergrid}, let us analyze which edges we have between distinct levels.
For this, it is convenient to observe the following connection between the radius $\rad_k$ of a vertex in level $ k $ and the disk radius $ r $, which holds due to the trigonometry of right triangles; refer to \cref{fig:regular_k-gons}.

\begin{observation}\label{obs:sinh_rk}
    For each $ 2 \leq k \leq r $, the radius $ \rad_k $ of a vertex in level $ k $ satisfies
    \[ \sin \left( \frac{\pi}{2^k} \right) = \frac{\sinh(r)}{\sinh(\rad_k)}. \]
\end{observation}

We divide the edge analysis into two parts:
First, we show that $ G_r $ is a \emph{strongly} hyperbolic uniform disk graph, with the root being a universal vertex adjacent to all other vertices.
And second, we consider the edges between any two levels defined by the regular polygons.

\begin{lemma}\label{lem:hypergrid_shudg}
    For each integer $ r \geq 2 $, the constructed graph $ G_r $ is a SHUDG.
\end{lemma}
\begin{proof}
    Recall that $G_r$ has $r$ levels.
    We thus need to show that the radius $ \rad_r $ of the vertices in level $ r $ is at most $ 2r $, which implies that they are adjacent to the root.
    Since $ \sinh(x) = (e^x - e^{-x}) / 2$ is monotonically increasing, this is equivalent to $ \sinh(\rad_r) \leq \sinh(2r) $.
    
    By \cref{obs:sinh_rk}, we need to compare the ratio of $ \sinh(r) $ and $ \sin(\pi/2^r) $ to $ \sinh(2r) $, which intuitively works out as $ \sinh(x) \approx e^x $ and $ \sin(x) \approx x $ yield $ \sinh(r) / \sin(\pi/2^r) \approx e^r / 2^{-r} = e^r \cdot 2^r \leq e^{2r} $.
    Indeed, using the exact definition of $\sinh$ and $\sin(x) \ge x / 2$ for $x \in [0, \pi/2]$, we obtain
      \begin{equation*}
        \sinh(\rad_r) 
        \overset{\text{Obs } \labelcref{obs:sinh_rk}}{=} 
            \frac{\sinh(r)}{\sin\left( \frac{\pi}{2^r} \right)}
        \le \frac{e^r - e^{-r}}{2} \cdot \frac{2 \cdot 2^r}{\pi}
        = \frac{e^r 2^r - e^{-r} 2^{r}}{\pi}
        \le \frac{ e^{2r} - e^{-2r}}{2} = \sinh(2r),
      \end{equation*}
      where the last inequality follows from $2^r \le e^r$, $2^r \ge e^{-r}$, and $\pi \ge 2$.
\end{proof}

To investigate which neighbors vertices have in addition to the root, let us first prepare some inequalities that are independent of our construction.

\begin{lemma}\label{lem:hyp-ineq}
    \newlength{\extraspace}
    \setlength{\extraspace}{0.3em}
    For all $ x \in \mathbb{R}_+ $ the following holds:
    \vspace{\extraspace / 2}
    \begin{enumerate}[(1)]
        \setlength{\itemsep}{\extraspace}
        \item \label{eq:cosh/sinh}
            $ \frac{\cosh(x)}{\sinh(x)} \geq 1 $
        \item \label{eq:cosh/sinh_2}
            $ \frac{\cosh(2x)}{\sinh^2(x)} \leq 2.1 $ for $ x \geq 2 $
        \item \label{eq:arccos}
            $ \sqrt{2x} \leq \arccos(1-x) \leq \pi \sqrt{\frac{x}{2}} $ for $ x \leq 2 $
        \end{enumerate} \vspace{\extraspace / 2} Moreover, (\ref{eq:cosh/sinh_2}) and (\ref{eq:arccos}) are tight up to constant factors.
        The same is true for (\ref{eq:cosh/sinh}) if $x$ is at least some constant.
\end{lemma}

\NewDocumentCommand{\ineqref}{sm}{%
    \IfBooleanT{#1}{\renewcommand{\crefpairconjunction}{,}}%
    \text{%
        \IfBooleanTF{#1}{Lem}{Lemma} 
        \labelcref{lem:hyp-ineq}\,(\labelcref{#2})%
    }%
    \renewcommand{\crefpairconjunction}{ and~}%
}
\begin{proof}
    The first inequality (\ref{eq:cosh/sinh}) (and its claimed tightness) follows directly from the definitions of $\cosh(x) = (e^x + e^{-x}) / 2$ and $\sinh(x) = (e^x - e^{-x}) / 2$.
    
    For the upper bound of inequality (\ref{eq:cosh/sinh_2}), it is more convenient to give a lower bound for the reciprocal.
    Using the definition of $\sinh$ and $\cosh$, we get
    \begin{equation*}
        \frac{\sinh^2(x)}{\cosh(2x)} =
        \frac{1}{2} \cdot \frac{(e^x - e^{-x})^2}{e^{2x} + e^{-2x}} =
        \frac{1}{2} - \frac{1}{e^{2x} + e^{-2x}}
    \end{equation*}
    As $e^{2x} + e^{-2x}$ is increasing and $x \ge 2$ by assumption, the above difference becomes minimal for $x = 2$, yielding $\sinh^2(x) / \cosh(2x) \ge 0.48$.  As the reciprocal of $0.48$ is less than $2.1$, the claim follows.  Tightness follows from the fact that the term is lower-bounded by $2$.
    
    Finally, the lower bound of (\ref{eq:arccos}) holds due to the series expansion of $ \arccos(1-x) $ at $ x = 0 $, which lower-bounds $ \arccos(1-x) $ for all $ x \in [0,2] $.
    For the upper bound, we have $ \sin(y) \geq 2y/\pi $ for $ y \in [0, \pi/2] $ and thus $ \sin^2(y) \geq 4y^2/\pi^2 $.
    Using the double-angle formula $ \cos(2y) = 1 - 2 \sin^2(y) $ and taking $ 2y = \pi \sqrt{\frac{x}{2}} $ for $ x \in [0, 2] $, we obtain 
    $ \cos(\pi \sqrt{\frac{x}{2}}) 
    = \cos(2y) 
    = 1 - 2 \sin^2(y) 
    \leq 1 - 8y^2/\pi^2
    = 1 - 8 (\frac{\pi}{2} \sqrt{\frac{x}{2}})^2 / \pi^2
    = 1 - x $,
    which yields the desired bound by taking the arccosine on both sides.  Tightness follows from the fact that the lower and upper bound match up to a constant factor.
\end{proof}

We now have all ingredients to count the number of neighbors a vertex in level $ i $ has in level $ j $.
For our purposes, we are mainly interested in an upper bound, that is, we have a guarantee that vertices we consider as non-adjacent indeed do not share an edge, whereas we might have some fewer edges than obtained from the upper bound.

\begin{lemma}\label{lem:Gr_neighbors}
    For every $ 2 \leq i \leq j \leq r $, every vertex of $ G_r $ in level $ i $ has at least two and at most $ 4 \sqrt{2^{j-i}} $ neighbors in level $ j $, 
    where the upper bound is tight up to a constant factor.
\end{lemma}

Before proving \cref{lem:Gr_neighbors}, let us discuss its implications to the structure of $ G_r $.
For this, consider a vertex $ v $ in some level $ i $ and the two closest vertices in level $ i + 1 $, one to the left and one to the right.
Note that these two vertices are connected to $ v $ since $ v $ has at least two neighbors.
Now let $ T $ denote the tree induced by the four edges from the root to the first level and the edges to the two closest neighbors for each vertex in levels $ 2, \dots, r - 1 $, see \cref{fig:Gr_neighbors}.
Observe that, apart from the root having degree 4, $ T $ is a spanning binary tree.
That is, the number of descendants of some vertex grows with base 2 in $ T $, whereas the number of neighbors in $ G_r $ grows with base $ \sqrt{2} $.
Hence, asymptotically, vertices that are adjacent in $ G_r $ have an ancestor-descendant relation in $ T $.
Additionally, for a vertex in level $ i $, there are far fewer neighbors than descendants in level $ j $, provided $ j - i $ is sufficiently large.
This observation is key to bounding the clique number of $ G_r $.
We remark that with a more careful analysis taking the separators shown in \cref{fig:regular_k-gons} into account, it can be shown that actually all edges between distinct levels obey the ancestor-descendant relation of $ T $, which we omit here in favor of a more elegant argument.

\begin{figure}
    \centering
    \includegraphics{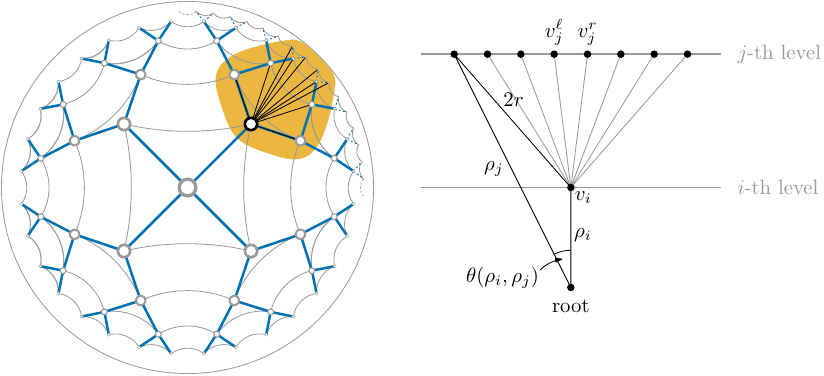}
    \caption{%
        Left: A spanning tree (blue) guiding the structure of $ G_r $, and the neighborhood of some vertex (black and orange).
        Note that the number of neighbors increases with the level, but slower than the number of descendants in the tree, whereas the angle in which a vertex has neighbors shrinks.
        Right: The number of neighbors a vertex $ v_i $ has in level $ j $ is determined by $ \theta(\rad_i, \rad_j) $, where $ \rad_i $ and $ \rad_j $ are the radii of the two levels.
        The angle $ \theta(\rad_i, \rad_j) $ is described by the shown triangle with side lengths $ \rad_i $, $ \rad_j $, and $ 2r $, which is captured by \cref{eq:max_angular_distance}.
    }
    \label{fig:Gr_neighbors}
\end{figure}

\begin{proof}
    Let $ v_i $ be some vertex in level $ i $.
    If $ i = j $, then $ v_i $ has exactly two neighbors in level $ j $ as every level induces a cycle, so assume $ i < j $.
    
    For the lower bound, consider the two vertices $ v^\ell_j, v^r_j $ in level $ j $ that have the smallest angular distance to $ v_i $ (\cref{fig:Gr_neighbors}).
    Observe that by symmetry, one of them is to the left of $ v_i $ and one to the right\footnote{The terms \emph{left} and \emph{right} should become clear from the picture.
      For a more formal definition, consider the line through the origin and $v_i$, oriented from the origin to $v_i$.
      This line splits the plane into a left and right half-plane.
      A neighbor of $v_i$ is a left or right neighbor depending on the half plane it lies in.}.
    By choice of $ v^\ell_j, v^r_j$, they are consecutive in the cycle induced by level $ j $, and thus are adjacent.
    To see that $ v^\ell_j, v^r_j$ are also adjacent to $ v _i $, consider a disk of radius $ 2r $ with center $ v^\ell_j $, which contains exactly the neighbors of $ v^\ell_j $, in particular $ v^r_j $ and the origin.
    Thus, the triangle formed by $ v^\ell_j, v^r_j $, and the origin is contained in the disk, and since $ v_i $ lies in this triangle, $ v_i $ is also contained in the disk.
    By symmetry, $ v^\ell_j, v^r_j $ are the two desired neighbors of $ v_i $ in level $ j $.

    For the upper bound, let $ \rad_i $ and $ \rad_j $ denote the radii of the vertices in level $ i $, respectively level $ j $.
    We count the number of left neighbors of $ v_i $ in level $ j $, which is determined by the angular distance $ \theta_{2r}(\rad_i, \rad_j) $ (\cref{eq:max_angular_distance}) such that two vertices with radius $ \rad_i $ and $ \rad_j $, respectively, have distance $ 2r $, and thus are adjacent.
    For brevity, we omit the subscript in the following and just write $\theta$ instead of $\theta_{2r}$.
    We refer again to \cref{fig:Gr_neighbors}.
    That is, we aim for an upper bound on $ \theta(\rad_i, \rad_j) $, which satisfies
    \begin{align*}
        \theta(\rad_i,\rad_j) 
        \overset{\labelcref{eq:max_angular_distance}}&{=}
            \arccos\left(\frac{\cosh(\rad_i)\cosh(\rad_j)-\cosh(2r)}{\sinh(\rad_i)\sinh(\rad_j)}\right) \\
        \overset{\text{Obs \labelcref{obs:sinh_rk}}}&{=} 
            \arccos \left(  
            \frac{\cosh(\rad_i)}{\sinh(\rad_i)}
            \frac{\cosh(\rad_j)}{\sinh(\rad_j)}
            - \frac{\cosh(2r)}{(\sinh(r))^2}
            \sin \left( \frac{\pi}{2^i} \right) 
            \sin \left(\frac{\pi}{2^j} \right)
        \right)
    \end{align*}
    by \cref{eq:max_angular_distance} and \cref{obs:sinh_rk}, the latter of which is applied to $ \sinh(\rad_k) $ for $ k = i,j $.
    We remark that, using the tightness of the bounds in \cref{lem:hyp-ineq}, it is not hard to see that upcoming inequalities we give are also asymptotically tight, yielding a tight upper bound.

    Since $ \arccos $ is monotonically decreasing, we aim to lower-bound its argument.
    The first part of the difference is lower-bounded by 1 due to \ineqref{eq:cosh/sinh}, and $ {\cosh(2r)}/{(\sinh(r))^2} $ is upper-bounded by $2.1$ with \ineqref{eq:cosh/sinh_2}, recall for the latter that $ r \geq 2 $.
    We conclude that
    \begin{align*}
        \theta(\rad_i,\rad_j) 
        \overset{ \ineqref*{eq:cosh/sinh,eq:cosh/sinh_2} }&{\leq} 
            \arccos\left( 
                1 - 2.1 \cdot \sin \left( \frac{\pi}{2^i} \right) 
                \sin \left(\frac{\pi}{2^j} \right) 
            \right) \\
        \overset{ \ineqref*{eq:arccos} }&{\leq }
            \pi \sqrt{
                \frac{2.1}{2} \cdot \sin \left( \frac{\pi}{2^i} \right) 
                \sin \left(\frac{\pi}{2^j} \right) 
            }\\
        &\leq \pi \sqrt{\frac{2.1}{2} \pi^2 \cdot 2^{-i} \cdot 2^{-j} }
        = \sqrt{\frac{2.1}{2}} \pi^2 \sqrt{2^{-(i+j)}},
    \end{align*}
    where the last inequality holds since $ \sin(x) \leq x $ for all $ x \geq 0 $.
    Note that we indeed may apply \ineqref{eq:arccos} since
    $ 
        2.1 \cdot \sin (\pi / 2^i) \sin(\pi / 2^j)
        \leq 2.1 \cdot \pi^2 \cdot 2^{-(i+j)}
        \leq 2.1 \cdot \pi^2 \cdot 2^{-(2+3)}
        \leq 2
    $.
    
    
    This immediately implies an upper bound on the number of left neighbors of $ v_i $ in level $ j $, and therefore also on the number $ \deg_j(i) $ of all neighbors in level $ j $.
    To see this, recall that the angular distance between two vertices in level $ j $ is $ 2 \pi / 2^j $.
    Hence, to have all neighbors within an angle of $ 2 \theta(\rad_i,\rad_j) $, we have
    \begin{equation*}
        \deg_j(i) 
        \leq \frac{2 \theta(\rad_i,\rad_j)}{2 \pi / 2^j} + 1
        \leq \frac{2 \sqrt{2.1/2} \pi^2 \sqrt{2^{-(i+j)}}}{2\pi/2^j} + 1
        = \sqrt{2.1/2} \pi \sqrt{2^{j-i}} + 1
        \leq 4 \sqrt{2^{j-i}},
    \end{equation*}
    where we recall for the last inequality that $ j - i \geq 1 $.
\end{proof}

Having understood which edges occur in our constructed graph $ G_r $, we can now derive a lower bound on the treewidth and an upper bound on the clique number.
The first is indeed easy to conclude from the lower bound of \cref{lem:Gr_neighbors}: 
Contracting each of the $ r $ levels yields an $ r $-clique since each two levels are connected by edges.

\begin{corollary}\label{cor:Gr_tw_logn}
    The treewidth of $ G_r $ is at least $ r \in \Omega(\log n) $, where $ n $ is the number of vertices of $ G_r $.
\end{corollary}

\todo[inline]{We have a proof sketch that $ G_r $ is far from having product structure in the sense that there are many vertices having large treewidth in their neighborhood. But we may as well just leave it out. In case we include it, we might bring back the discussion at the end \cref{sec:lower_bound_tw_loglogn}}

Finally, we obtain an upper bound on the clique number from the upper bound of \cref{lem:Gr_neighbors}.
For this, we exploit that the number of neighbors a vertex on level~$ i $ has in level~$ j $ grows only with base~$ \sqrt{2} $, whereas the number of vertices in a level grows with base~2.
This enables us to conclude the following relation between vertices of triangles.
As any three vertices in a clique form a triangle, the following lemma is key to bounding the clique number.

\begin{lemma}\label{lem:Gr_triangle}
    For every triangle in $ G_r $ with vertices $ v_i $, $ v_j $, and $ v_k $ on levels $2 \leq i \leq j \leq k$, respectively, we have $ k - j \leq j - i + 6 $.
\end{lemma}

\begin{figure}
    \centering
    \includegraphics{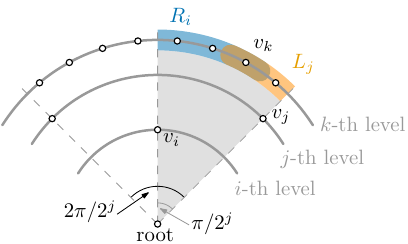}
    \caption{%
        A triangle of $ G_r $, where $ v_k $ lies both in the neighborhood of $ v_i $ (blue) and of $ v_j $ (orange).
        We conclude that the right half of the neighborhood of $ v_i $ plus the left half of the neighborhood of $ v_j $ is larger than the number of vertices in level $ k $ that lie between $ v_i $ and $ v_j $ (gray wedge).
    }
    \label{fig:Gr_triangle}
\end{figure}

\begin{proof}
    We make use of \cref{lem:Gr_neighbors} to obtain the number of neighbors $ v_i $ and $ v_j $ have in level $ k $ and compare the sum with the number of vertices in level $ k $ that lie between $ v_i $ and $ v_j $.
    Refer to \cref{fig:Gr_triangle} for this.
    Without loss of generality, we may assume that $ v_i $ is to the left of $ v_j $.
    Since $ v_i $ and $ v_j $ have a common neighbor in level $ k $, we have that the set $R_i$ of right neighbors of $ v_i $ intersects the set $L_j$ of left neighbors of $ v_j $ in level $ k $.
    Hence, we have $ |R_i| + |L_j| > n_k $, where $n_k$ denotes the number of vertices in level $ k $ lying between $ v_i $ and $ v_j $.
    
    By \cref{lem:Gr_neighbors}, we have $ |R_i| \leq 2 \sqrt{2^{k-i}} $ and $ |L_j| \leq 2 \sqrt{2^{k-j}} $.
    Since $ i \leq j \leq k $, we have $ |R_i| \geq |L_j| $, and thus we can bound the sum $ |R_i| + |L_j| \leq 2 |R_i| \leq 4 \sqrt{2^{k-i}} $.
    
    To bound $ n_k $, assume $ j < k $ as there is nothing to show otherwise.
    Recall that the angular distance between two vertices in level $ j $ is $ 2\pi / 2^j $, and thus the angular distance between $ v_i $ and $ v_j $ is at least $ \pi / 2^j $ by construction.
    Within this angle, we have at least $ (\pi / 2^j) / (2\pi / 2^k) = 2^{k-j}/2 $ vertices in level $ k $, where the denominator is the angular distance between two vertices in level $ k $.
    
    To sum up, we have 
    $ 2^{k-j}/2 \leq n_k < |R_i| + |L_j| \leq 4 \cdot 2^{(k-i)/2} $,
    implying $ k - j < (k - i)/2 + 3 $, which is equivalent to the desired inequality $ k - j < j - i + 6 $.
\end{proof}

It follows from \cref{lem:Gr_triangle} that if a clique in $G_r$ has vertices in levels $ i $ and $ j \geq i $ with $ \Delta = j - i $, then all further vertices of the clique in larger levels are in levels $ j, \dots, j + \Delta + 6 $.
Applying this to all pairs of vertices in a clique, we obtain that the distances between the levels shrink exponentially, where the \emph{distance between two levels $ i $ and $ j $} refers to $ |i - j| $.
We formalize this in the following lemma.

\begin{lemma}\label{lem:Gr_clique-levels}
    Let $ r \geq 1 $, $ c \geq 0 $ be integers, and let $ S \subseteq [r] $ be a set such that for every three elements $ i < j < k $ of $ S $, we have $ k - j \leq j - i + c $.
    Then $ |S| < \log_2(r) + c + 2 $.
\end{lemma}

\begin{proof}
    Let us first describe how an optimal set $ S^* $ looks like before arguing that this is indeed optimal.
    We choose the elements for $ S^* $ greedily, adding always the largest element such that $ k - j \leq j - i + c $ is maintained for any three elements $ i < j < k $ in~$ S^* $.
    We claim that $ S^* $ contains roughly the $ c $ largest elements and starting from there, leaves exponentially growing gaps when smaller elements are added.
    More formally, we define $ S^* $ to contain the elements $ r, \dots, r - c $ and the elements $ s_a^* = r - c - 2^a $ for $ a = 0, \dots, \lfloor \log_2(r - c - 1) \rfloor $, and then show that $ S^* $ indeed results from choosing the elements greedily.
    To show that $ S^* $ is a valid choice for $ S $, it suffices to compare each two consecutive elements with the largest element $ r $, where consecutive refers to two elements $ s, s' \in S^* $ such that there is no third element $ x \in S^* $ with $ s < x < s' $.
    Indeed, the largest $ c + 2 $ elements may be contained in $ S^* $ as the three elements $ r, r-c, r - c - 2^0 $ satisfy $ r - (r - c) = c \leq c+1 = (r-c) - (r-c - 2^0) + c $.
    Similarly, for $ a > 0 $, we have 
    $ r - s_{a-1}^* 
    = r - (r - c - 2^{a-1}) 
    \leq 2^{a-1} + c
    \leq 2^{a} - 2^{a-1} + c
    \leq (r - c - 2^{a-1}) - (r - c - 2^{a}) + c
    = s_{a-1}^* - s_{a}^* + c$.
    Since the largest $ c + 2 $ elements form an interval and for the $ s_a^* $, all inequalities are tight, none of the elements can be chosen larger without violating the requirements.
    Thus, $ S^* $ indeed is the result of a greedy approach.
    Moreover, the resulting set has size $ |S^*| = c + 1 + \lfloor \log_2(r - c - 1) \rfloor + 1 < \log_2(r) + c + 2 $.

    It remains to show that there is no larger set $ S $ meeting the requirements. 
    To do so, consider all subsets of $ [r] $ with maximum cardinality satisfying $ k - j \leq j - i + c $ and let $ S $ be the one that maximizes the sum of its elements.
    We show that $ S = S^* $, which means that $ S^* $ indeed has maximum cardinality.
    Assume for contradiction that $S \neq S^*$ and let $ t \in [r] $ be the largest integer such that $ S $ and $ S^* $ disagree about $ t $. 
    Since $ S^* $ is constructed greedily from large to small elements and the two sets agree on all larger elements, we have that if $ t \in S $, then also $ t \in S^* $.
    Hence, $ t \in S^* $ but $ t \notin S $.
    
    We finish the proof by moving some elements of $ S $ so that the sum gets larger, which contradicts the choice of $ S $.
    First, if $ t = r $, then add 1 to all elements of $ S $.
    This does not change the differences, and thus yields a valid set of the same cardinality whose sum is larger.
    So assume $ t < r $ and let $ t' \in S $ be the largest integer with $ t' < t $.
    Note that $ t' $ exists as $ S $ is of maximum cardinality and thus at least as big as $ S^* $.
    Now, we claim that the set $ S' $ obtained from $ S $ by replacing $ t' $ with  $t$ satisfies $ k - j \leq j - i + c $ for every $ i < j < k $.
    We only need to check triples containing $ t $.
    Also recall that we may assume $ k = r $ as this maximizes $ k - j $.
    First, for $ j = t $, the difference between $ i $ and $ j $ is now even larger, while the difference between $ j $ and $ k $ shrinks, so $ k - j = k - t < k - t' \leq t' - i < t - i = j - i $ for every $ i < t' < t$ in $ S $, and thus also in $ S' $.
    Second, if $ i = t $, then the $ i, j, k $ are all contained in $ S^* $ by the choice of $ t $ and thus satisfy the requirement for $ S' $.
    We conclude that $ S = S^* $ and therefore $ |S| = |S^*| < \log_2(r) + c + 2 $.
\end{proof}

Putting everything together, we conclude that the clique number of $ G_r $ is indeed $ O(\log \log n) $.

\begin{lemma}\label{lem:Gr_clique_loglogn}
    The clique number of $ G_r $ is at most $ O(\log r) = O(\log \log n) $, where $ n $ is the number of vertices of $ G_r $.
\end{lemma}

\begin{proof}
    Consider a clique $ C $ in $ G_r $ whose size we aim to bound.
    First recall that each level induces a cycle and thus has clique number 2.
    Hence, up to a factor of 2, we may assume that the vertices of $ C $ are in pairwise distinct levels.
    Now each three vertices of $ C $ in levels $ i < j < k $ form a triangle and thus satisfy $ k - j \leq j - i + 6 $ by \cref{lem:Gr_triangle}.
    Then, \cref{lem:Gr_clique-levels} shows that $ C $ contains only $ 2 (\log_2(r) + 8) \in O(\log r) = O(\log \log n) $ vertices. 
\end{proof}
Finally, \cref{cor:Gr_tw_logn,lem:Gr_clique_loglogn} together prove \cref{thm:hypergrid}.

\subsection{Proof of \texorpdfstring{\cref{thm:sub-hypergrid}}{Theorem 9}}
\label{sec:lower_bound_tw_loglogn}
\todo{adjust theorem number in texorpdfstring}

The goal of this section is to prove \cref{thm:sub-hypergrid}, i.e., to construct SHUDGs with disk radius $ r \in \Theta(\log n) $, clique number $ O(1) $, and treewidth $ \Omega(\log \log n) $, as this implies \cref{main:noPS}.
We do this by taking a subgraph of the graph $ G_r $ for $ r \geq 8 $ constructed in \cref{sec:lower_bound_tw_logn} that has constant clique number but unbounded treewidth.
For this, let $ G_r' $ be the subgraph of $ G_r $ that is induced by the root and the vertices in levels $ 2^k $ for $ k = 3, \dots, \log_2(r) $.
Recall that $ G_r $ has $ r $ levels, where the $ k $-th level contains $ 2^k $ vertices for $ k \geq 2 $, and thus the $ r $-th level contains more than half of the vertices of $ G_r $.
Thus, $ G_r' $ has $ \Theta(\log r) = \Theta( \log \log n ) $ levels and $ \Theta(n) $ vertices, where $ n $ is the number of vertices of $ G_r $.
Further recall that in $ G_r $, and thus also in $ G_r' $, \cref{lem:Gr_neighbors} shows that each two levels are connected by edges.
Hence, we again obtain a clique-minor by contracting each level.

\begin{lemma}\label{lem:Gr'_tw_loglogn}
    The treewidth of $ G_r' $ is at least $ \Omega(\log r) = \Omega(\log \log n) $, where $ n $ is the number of vertices of $ G_r' $.
\end{lemma}

Recall that with the root being adjacent to all vertices, this already shows that $ G_r' $ does not admit product structure since graphs with product structure do not have vertices with large treewidth in their neighborhood by \cref{lem:tw_neighborhood}.
\thomas{remark for shortening: I think we can just drop the previous sentence.}%
That is, the hard part is again to bound the clique number, for which we heavily rely on our findings for $ G_r $.
For this, recall that the distance between two levels $ i $ and $ j $ refers to $ |i - j| $, and not, e.g., to the hyperbolic distance.

For the main idea, recall from \cref{lem:Gr_triangle,lem:Gr_clique-levels} that in $ G_r $, the distances between the levels used by a clique shrink exponentially from the root to the outermost level, i.e., most vertices of any clique are in the outer levels.
In contrast, we choose the levels for $ G_r' $ with exponentially growing distances, i.e., mostly inner levels are chosen.
Thus, for every clique in $G_r$, the levels chosen for $G_r'$ hit only very few vertices of the clique, which is shown in the following lemma.

\begin{lemma}\label{lem:Gr'_clique}
    The clique number of $ G_r' $ is at most $5$.
\end{lemma}

\begin{proof}
    The root, which we denote by $ v_1 $, is adjacent to all other vertices, so we show that the clique number of $ G_r' - v_1 $ is at most 4.
    Since each level induces a cycle, which has clique number 2, it suffices to show that every clique uses at most two levels.
    So, assume there is a triangle in $ G_r' - v_1 $ with vertices in pairwise distinct levels for the sake of contradiction.
    Let $ 2^i, 2^j, 2^k $ denote the levels of $ G_r $ corresponding to the three levels of the triangle in $ G_r' - v_1 $, where $ 3 \leq i < j < k $.
    On the one hand, the distance between the the larger two levels is $ 2^k - 2^j \geq 2^j $.
    On the other hand, we have $ 2^j - 2^i \leq 2^j - 2^3 < 2^j - 6 $.
    Together, this yields $ 2^k - 2^j \geq 2^j > 2^j - 2^i + 6 $,
    contradicting \cref{lem:Gr_triangle}, which shows that $ 2^k - 2^j \leq 2^j - 2^i + 6 $ for a triangle on levels $ 2^i, 2^j, 2^k $.
\end{proof}
This concludes the proof that hyperbolic uniform disk graphs do not admit product structure as \cref{lem:Gr'_tw_loglogn,lem:Gr'_clique} together show \cref{thm:sub-hypergrid}, which in turn proves \cref{main:noPS}.

%
%

\subsection{Dependence on the Disk Radius}
\label{sec:disk-radius}


We finish the section with a discussion on the disk radius.
The graph families we construct for \cref{thm:hypergrid,thm:sub-hypergrid} have disk radius $ \Theta(\log n) $.
In terms of product structure, this already shows that the class of (S)HUDGs does not have product structure, even when restricted to graphs with constant clique number.
However, it raises the natural question of whether this remains true when restricting the class further to larger or smaller disk radius.
In the following, we first discuss the regime where $r \in o(\log n)$.
Afterwards, we consider the regime where the disk radius $r$ is asymptotically larger than $\log n$.
In both cases, we extend the previous construction, giving lower bounds on the treewidth and row-treewidth, which proves the second part of \cref{main:rPS} and strengthens \cref{main:treewidth}.
In particular, we show that there is no product structure unless the disk radius is in $ O(1) $.
For the missing regime of $r \in O(1)$ (including radii shrinking with $n$), we show the contrary in Section~\ref{sec:upper_bound}, i.e., for such small radii, the resulting graph classes indeed have product structure.

%

\subparagraph{Smaller Disk Radius.}

For smaller disk radius $r$, we can adapt \cref{thm:hypergrid,thm:sub-hypergrid} by simply adding isolated vertices.
This increases the number of vertices while leaving the disk radius unchanged, so the disk radius grows slower with respect to the number of vertices.
Note that if $ r $ is super-constant the treewidth in the neighborhood of the root, and thus the row-treewidth, is still unbounded as it grows linearly, respectively logarithmically, with the disk radius $ r $ by \cref{cor:Gr_tw_logn,lem:Gr'_tw_loglogn}.
Similarly, the clique number does not change by adding isolated vertices and thus is still bounded by $ O(\log r) $, respectively $ O(1) $, by \cref{lem:Gr_clique_loglogn,lem:Gr'_clique}.

\begin{corollary}
    \label{cor:lower-bound-sub-log-radius}
    For every $r \in O(\log n)$, there are families of $n$-vertex HUDGs with disk radius in $\Theta(r)$, 
    \begin{itemize}
        \item clique number in $ O(\log r) $, and (row-)treewidth in $ \Omega(r) $, respectively
        \item clique number in $ O(1) $, and (row-)treewidth in $ \Omega(\log r) $.
    \end{itemize}

\end{corollary}

Note that this extends \cref{thm:hypergrid,thm:sub-hypergrid} to $ r \in o(\log n) $ at the cost of a worse lower bound on the row-treewidth in terms of $ n $ and of having only HUDGs instead of SHUDGs.
Let us briefly discuss why both restrictions are necessary.
First, in \cref{sec:upper_bound}, we give an upper bound on the row-treewidth depending on the disk radius.
Thus, if the disk radius tends to a constant, so does the row-treewidth.
Second, we need to consider HUDGs instead of SHUDGs here for the following reason.
Every SHUDG with radius $ r $ can be covered with $\max\{8, 2\pi e^r\}$ cliques \cite[Lemma~8]{blasius_strongly_2023}.
Thus, for a family $\calG$ of SHUDGs with disk radius $r \in o(\log n)$, the graphs in $\calG$ can be covered with $o(n)$ cliques, implying super-constant clique number.  Thus, to get constant clique number, we have to consider HUDGs instead of SHUDGs.

\subparagraph{Larger Disk Radius.}

We continue with the case where the disk radius $r$ grows super-logarithmically with $n$.
In this case, the construction of $G_r$ can be easily adapted by removing vertices of $G_r$ from the current outermost level until the desired dependency between the number of vertices $n$ and the disk radius $r$ is reached.
Note that with increasing $ r $, the number of remaining vertices still grows but slower than the $ \Theta(2^r) $ from the initial construction.
However, the number of vertices keeps growing exponentially with the number of levels.
The property of being a SHUDG is inherited both for the modified $ G_r $ and for its subgraph induced by the levels chosen for $ G_r' $.
The same holds for the upper bound on the clique size which is logarithmic in the number of levels, i.e., $ O(\log \log n) $, respectively constant.
Similarly, the treewidth is again the number of levels as they can be contracted to form a clique-minor.
This yields the following corollary.

\begin{corollary}
    \label{cor:lower-bound-super-log-radius}
    For every $r \in \Omega(\log n)$, there are families of $n$-vertex SHUDGs with disk radius in $\Theta(r)$, 
    \begin{itemize}
        \item clique number in $O(\log\log n)$, and (row-)treewidth in $\Omega(\log n)$, respectively
        \item clique number in $O(1)$, and (row-)treewidth in $\Omega(\log\log n)$.
    \end{itemize}

\end{corollary}

This gives the same lower bounds for the row-treewidth as we have in the $ \Theta(\log n)$ setting.
We note that one might expect that the row-treewidth should grow with the radius as the graphs become less Euclidean.
However, we believe that $r \in \Theta(\log n)$ is a natural breaking point beyond which increasing $r$ further does not have big impact.
Specifically for our result, we remark that if $r $ is at least constant, then the treewidth, and thus also row-treewidth, is in $O(\omega \log n)$ for clique number $ \omega $~\cite{Struc_Indep_Hyper_Unifor_Disk_Graph-Blaesius25,Treew_Graph_Balan_Separ-DvoraNorin19}.
Thus, our lower bound of $\Omega(\log n)$ with $\omega \in O(\log\log n)$ is tight up to a factor of $\log\log n$, which leaves little room for improvement.

More generally, we are not even sure whether increasing the disk radius beyond $\Theta(\log n)$ even changes the graph class and would in fact conjecture that it does not.
The reasoning for this is as follows.
The main benefit of increasing the radius is that there is more space to fit non-adjacent vertices in the neighborhood of a single vertex.
However, $r \in \Theta(\log n)$ suffices for a star on $n$ vertices~\cite{Struc_Indep_Hyper_Unifor_Disk_Graph-Blaesius25}, i.e., having additional space does not help to place additional non-adjacent vertices in the neighborhood.
In contrast to this, the main disadvantage of increasing the disk radius seems to be that one can place less vertices inside a cycle that are not adjacent to the cycle.
However, when $r \in \Theta(\log n)$, it is already impossible to fit a single non-adjacent vertex inside a cycle of length $n$.
Thus, increasing $r$ beyond $\Theta(\log n)$ seems to neither allow for additional graphs nor exclude graphs that work for smaller disk radii.

\todo[inline]{At the moment, we mostly omit the tightness of our constructions. Do we want to discuss this? → include more on lower bounds}



\section{Product Structure for Small Disk Radius}
\label{sec:upper_bound}

\newcommand{\upperBound}{3^{8r}}


In this section, we show that hyperbolic uniform disk graphs whose disk radius and clique number are (sub-)constant admit product structure, which finishes the proof of \cref{main:rPS}.
In fact, we show that every HUDG $ G $ with disk radius $ r $ and clique number $ \omega $ is a subgraph of $ H \boxtimes K_k $, where $ H $ is a hyperbolic tiling and $ k \in O(\omega \cdot \upperBound) $, which emphasizes their grid-structure.
Since $ H $ is planar, and planar graphs are subgraphs of $ H' \boxtimes P \boxtimes K_3 $~\cite{dujmovic_planar_2019}, it follows that $ G \subseteq H' \boxtimes P \boxtimes K_{3k} $, where $ H' $ has treewidth $ 3 $ and $ P $ is a path.
We remark that with a more careful analysis, the treewidth of $ H' $ can be improved to 2.
\todo{maybe do this}%
Recall that this implies that for every such $ G $, there is a graph $ H'' $ of treewidth $ O(\omega \cdot \upperBound) $ such that $ G \subseteq H'' \boxtimes P $, namely $ H'' = H' \boxtimes K_{3k} $.
Thus, the row-treewidth of $ G $ is at most $ O(\omega \cdot \upperBound) $.
That is, in the following theorem the product $ H \boxtimes K_{O(\omega \cdot \upperBound)} $ is the strongest statement and the others follow.

\begin{theorem}\label{thm:upper_bound}
    For every HUDG $ G $ with disk radius $ r $ and clique size $ \omega $ it holds that $ G \subseteq H \boxtimes K_{O(\omega \cdot \upperBound)} \subseteq H' \boxtimes P \boxtimes K_{O(\omega \cdot \upperBound)}$, where $ H $ is a (possibly irregular) hyperbolic tiling and $ H' $ is a graph of treewidth at most $ 3 $.
    Moreover, the row-treewidth of $ G $ is in $ O(\omega \cdot \upperBound) $.
\end{theorem}

Note that \cref{thm:upper_bound} indeed implies the first part of \cref{main:rPS}
since the row-treewidth only depends on $ \omega $ and $ r $, and thus we obtain product structure for every graph class with (sub-)constant clique number and disk radius.
Further note that we also obtain an upper bound on the row-treewidth if clique number or disk radius are super-constant for some family of HUDGs.
However, in these cases the upper bound on the row-treewidth also grows and thus we do not obtain product structure, which is no surprise given \cref{sec:lower_bound}.


To find the product $ H \boxtimes K_k $, it is useful to interpret the multiplication with a complete graph in the language of quotients.
For this, recall that $ H \boxtimes K_k $ consists of copies of $ K_k $, one for each vertex of $ H $, that are joined by complete bipartite graphs according to the edges of $ H $.
This defines a partition \calP whose parts correspond to the copies of $ K_k $.
Then, the quotient $ (H \boxtimes K_k) / \calP $, i.e., the graph obtained by contracting each part of \calP, is $ H $.
Moreover, every graph also admitting a partition into parts of size at most $ k $ such that the quotient is (a subgraph of) $ H $ is a subgraph of $ H \boxtimes K_k $.
That is, to show that a graph $ G $ is a subgraph of $ H \boxtimes K_k $, we aim to find a partition \calP of $ G $ into parts of size at most $ k $ such that the quotient is $ H = G / \calP $.

This observation is also used by Dvořák, Huynh, Joret, Liu, and Wood~\cite{dvorak_notes_2021} to show that Euclidean unit disk graphs with constant clique number admit product structure.
Although they use a Euclidean $ \{ 4, 4 \} $-tiling, for our purposes it is more helpful to briefly sketch their argument using a Euclidean $ \{ 6, 3 \} $-tiling, i.e., three 6-gons meet at each vertex.
Given an embedding of the vertices of a EUDG $G$ with clique number $ \omega $, tile the Euclidean plane with a regular tiling into hexagons, and consider the resulting partition $\calP$ of the graph.
Now let $ Q = G / \calP $ be the quotient, i.e., the graph obtained by contracting the vertices in each hexagon.
By choosing the size of the hexagons carefully, a trade-off between two properties can be achieved: 
First, each hexagon can be covered by constantly many unit disks and thus contains $ O(\omega) $ vertices.
And second, adjacent vertices in $ G $ are in the same or in adjacent tiles, so the quotient $ G $ is a subgraph of the dual $ H $ of the hexagonal tiling, which admits product structure~\cite{dujmovic_planar_2019}.
Hence, $ G \subseteq Q \boxtimes K_{O(\omega)} \subseteq H \boxtimes K_{O(\omega)} $ also admits product structure.

We lift this approach to the hyperbolic setting with two major changes.
Not surprisingly, we replace the Euclidean tiling by a hyperbolic tiling to adjust to the geometry.
However, we also need to face different disk radii since, in contrast to \cite{dvorak_notes_2021}, we cannot freely choose the size of our tiles.
For large disk radii, we solve this by merging tiles so that we obtain an irregular tiling with larger tiles, whereas we subdivide tiles for small disk radii as needed.
To do so, we provide a geometric interpretation of what we aim for and discuss more detailed why this implies \cref{thm:upper_bound}.
Two tiles are called \emph{adjacent} if they are adjacent in the dual, i.e., if they share an edge.


\tiling*

Observe that this indeed implies that every HUDG with radius $ r $ and clique number $ \omega $ is a subgraph of $ H \boxtimes K_k $, where $ H $ is a (possibly irregular) tiling of the hyperbolic plane, namely the dual of the tiling from \cref{lem:tiling}, and $ k \in O(\omega \cdot \upperBound) $.
To see this, consider the partition \calP with one part for the vertices of each tile, where we may assume without loss of generality that no vertex lies on the boundary of some tile.
By the first property, each part contains at most $ O(\omega \cdot \upperBound) $ vertices since each disk of radius $ r $ covers a clique and thus at most $ \omega $ vertices.
Thus, we have $ G \subseteq Q \boxtimes K_k $ for the quotient $ Q = G / \calP $ and some $ k \in {O(\omega \cdot \upperBound)} $.
Then, by the second property, the endpoints of every edge lie in the same or in adjacent tiles, so the quotient $ Q $ is a subgraph of the dual $ H $ of the tiling from \cref{lem:tiling}.
Together, we obtain $ G \subseteq Q \boxtimes K_k \subseteq H \boxtimes K_k $, as required by \cref{thm:upper_bound}.

The main challenge of this section is to prove \cref{lem:tiling}, which we divide into two parts.
The first part handles large radii and starts with a regular $ \{7, 3\} $-tiling whose tiles are merged until the second property is satisfied.
Here, the tiles have a rapidly growing maximum degree to meet the requirements of the hyperbolic plane.
The case for small radii then starts with a regular $ \{ 4, 5 \} $-tiling, where we subdivide the tiles until they are small enough for the first property.
This results in a tiling that gets more Euclidean the smaller $ r $ is, which we discuss more detailed in \cref{sec:tiling_small_radius}.
The threshold between the two cases is $ d \approx 0.53 $, where $ 2d $ is the distance between two opposite sides of a tile in a $ \{ 4, 5 \} $-tiling; see \cref{lem:distances-in-tilings}.

\subsection{Tiling for \texorpdfstring{$ \bm{r \geq d} $}{r > d}}
\label{sec:tiling_large_radius}

%

We first prove \cref{lem:tiling} for $ r \geq c $, with $c$ being the circumradius of a tile in a regular $\{7, 3\}$-tiling.
The gap between $ d \approx 0.53 $ and $ c \approx 0.62 $ (see \cref{lem:distances-in-tilings}) is then closed at the end of the subsection.
Before we get into the details, we sketch the overall ideas.
We start with a $\{7, 3\}$-tiling (recall \cref{fig:prelim-tiling}), whose tiles we aim to merge into larger tiles satisfying the conditions of the lemma.
To distinguish between the two, we call the tiles of the $ \{ 7, 3 \} $-tiling \emph{small tiles}, and the tiles of our constructed tiling \emph{large tiles}.
Note that the condition $r \ge c$ ensures that each small tile can be covered with a disk of radius $r$.
Hence, for the first condition of \cref{lem:tiling}, we only need to make sure that not too many small tiles are merged into a large tile.
To meet the second condition, we have to construct our tiling so that crossing a large tile requires distance more than $2r$.
We refer to \cref{fig:merge_tilings_overview} for an illustration of the result.

\begin{figure}
    \centering
    \includegraphics{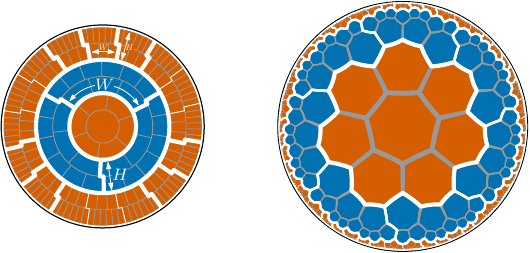}
    \caption{%
        To ensure that non-adjacent tiles are sufficiently far from each other, we merge 7-gons as indicated by the white lines.
        High-level, we define rings of height $ H \approx 2r $, which we cut into parts of width $ W \in O(3^{4r}) $ (left).
        Locally, however, we stick to the 7-gons from a $ \{ 7, 3 \} $-tiling (right).
    }
    \label{fig:merge_tilings_overview}
\end{figure}

To specify how we merge small tiles into large tiles, we consider BFS-layers in the dual $D$ of the $\{7, 3\}$-tiling.
Note that each vertex in $D$ corresponds to a small tile and recall from \cref{sec:preliminaries} that the dual $D$ of the $\{7, 3\}$-tiling is a $\{3, 7\}$-tiling; also see \cref{fig:3_7_tiling_large_radius}.
Now choose an arbitrary root vertex $r$ in $D$ and consider BFS-layers from $r$.
The $\ell$-th layer contains the vertices with distance $\ell$ from $r$.
Observe that each layer induces a cycle (except the $0$-th layer, which contains only the root).
Thus, every vertex $v$ in a layer $\ell \ge 1$ has one clockwise and one counter-clockwise neighbor on the same layer, which we call \emph{right} and \emph{left sibling}, respectively.
Moreover, $v$ has one or two consecutive neighbors in layer $\ell - 1$, which we call \emph{(left/right) parent}, and three or four consecutive neighbors in layer $\ell + 1$, which we call \emph{children}.
Observe that the left-most child of $v$ coincides with the right-most child of its left sibling.
A path in $D$ is called \emph{horizontal path} if it contains only vertices from a single layer.
It is a \emph{vertical path} if it contains at most one vertex from every layer.
If $u$ and $v$ are the endpoints of a vertical path with $u$ having smaller layer, then $u$ is an \emph{ancestor} of $v$ and $v$ is a \emph{descendant} of $u$.
The \emph{left/right-most} descendant and ancestor on a fixed layer are defined in the canonical way.

\begin{figure}
    \centering
    \includegraphics{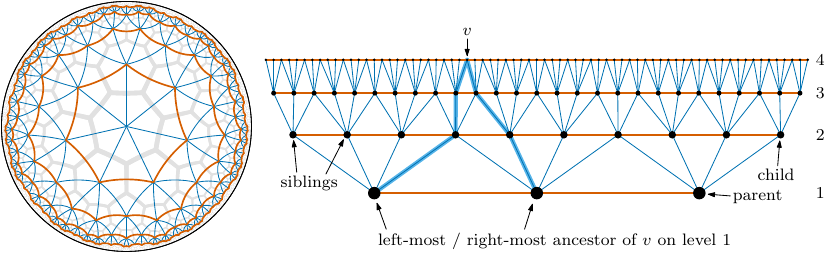}
    \caption{Left: Illustration of the $\{3, 7\}$-tiling (dual of the $\{7, 3\}$-tiling) with colors indicating the BFS layers.
      Right: Parts of the $\{3, 7\}$-tiling with BFS-layers as horizontal lines.
      The two highlighted paths are vertical paths.
      Paths consisting of only red edges are horizontal paths.}
    \label{fig:3_7_tiling_large_radius}
\end{figure}

With this, we can now specifying the large tiles; see Figure~\ref{fig:3_7_tiling_large_radius_large_tiles}.
We start by defining \emph{rings} of height $H$ where the $k$-th ring consists of the BFS-layers $H \cdot k, \dots, H \cdot (k + 1) - 1$ for $k \ge 0$.
The $0$-th ring alone forms one large tile.
For larger $k$, we split the $k$-th ring into pieces, each of which again forms a large tile.
For this, consider for each $k \ge 1$ the smallest BFS-layer of the $k$-th ring, i.e., the $(H \cdot k)$-th layer.
Now cut the cycle induced by this layer into paths consisting of $W$ vertices each.
If the length of the cycle is not divisible by $W$, we get one longer path with up to $2\cdot W$ vertices.
Then, we assign each of the remaining vertices of the $k$-th ring to the path containing its ancestor(s).
For now, each of these paths together with its assigned vertices forms a large tile (we need to merge few of the large tiles in the same ring in a moment).

\begin{figure}
    \centering
    \includegraphics[page=2]{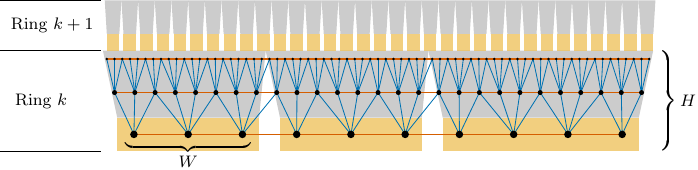}
    \caption{Each ring consists of the same number $H$ of layers.
      The paths consisting of $W$ vertices in the smallest layer of the $k$-th ring are highlighted in orange (the last path is slightly longer as 10 is not divisible by 3).
      The assignment of the other vertices in the ring is illustrated in gray.
      The paths together with the assigned vertices form the large tiles.
      As the bottom layer of the $(k + 1)$-st ring contains much more vertices, there are more paths and thus more large tiles.}
    \label{fig:3_7_tiling_large_radius_large_tiles}
\end{figure}

Concerning the assignment to paths, note that a vertex $v$ from a layer larger than $H \cdot k$ can have multiple ancestors in layer $H \cdot k$.
For tie-breaking, we in this case favor the right-most ancestor of $v$ in layer $H \cdot k$.
We will see in a moment that $v$ has at most two ancestors in each layer, so this does not make a big difference; also see \cref{fig:3_7_tiling_large_radius_large_tiles}.

Our plan now is the following.
First, based on $H$ and $W$, we can count the number of small tiles that are merged into a large tiles, giving an upper bound on the number of disks needed to cover a large tile.
Secondly, we have to show that non-adjacent large tiles are distant; see \cref{fig:tiling_large_r_merging_large_tiles}.
For this, we show that two large tiles separated by a ring are sufficiently far apart if $H$ is large enough.
We show the same for non-adjacent large tiles on the same ring if $W$ is sufficiently large.

\begin{figure}
    \centering
    \includegraphics{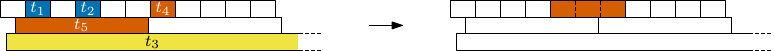}
    \caption{Illustration of the different types of non-adjacent large tiles. Two blue tiles ($t_1$, $t_2$) are distant if $W$ is sufficiently large.
      A blue tile and a yellow tile ($t_3$) are distant if $H$ is large enough.
      The red tiles ($t_4$, $t_5$) are non-adjacent but maybe too close.
      We fix this by merging three consecutive large tiles as highlighted right (which makes the too close tiles adjacent).}
    \label{fig:tiling_large_r_merging_large_tiles}
\end{figure}

After that, there only remain non-adjacent large tiles in neighboring rings.
For this, we have to adjust the tiling slightly, because right now, two non-adjacent tiles (not sharing an edge) could share a vertex, allowing for arbitrarily close points.
To resolve this, consider two vertices $u$ and $v$ in the outer-most layer of the $(k - 1)$-st ring (layer $H \cdot k - 1$) that are adjacent but belong to different large tiles.
Then $u$ and $v$ are siblings and thus have a common child in layer $H \cdot k$ that belongs to the $k$-th ring.
Let $w$ be this child.
We merge the large tile containing $w$ with the large tile to its left and right in the same ring, i.e., the actual large tile consists of a triple of the original large tiles.

This concludes the construction and we can now follow the plan outlined above, starting with counting the number of small tiles in large tiles large tiles.
For this, the following observation will be helpful.

\begin{lemma}
    \label{lem:3_7_tiling_exponential_growth}
    Let $P$ be a horizontal path of $k$ vertices in layer $\ell \ge 1$ of the $\{3, 7\}$-tiling.
    The vertices of $P$ have between $2k + 1$ and $3k + 1$disjoint children.
\end{lemma}
\begin{proof}
    Each vertex of $P$ has three or four children.
    As every two consecutive vertices in the path share one child, we get at least $2k + 1$ and at most $3k + 1$ disjoint children.
\end{proof}

Now we can count the number of small tiles in a large tile.

\begin{lemma}
    \label{lem:tiling_large_r_count_small_tiles}
    Each large tile contains at most $O(3^H \cdot W)$ small tiles if $H \ge 3$.
\end{lemma}
\begin{proof}
    We first count the number of small tiles in a large tile that is not one of the merged triples.
    The tile's path in the smallest layer contains $W$ vertices.
    Due to \cref{lem:3_7_tiling_exponential_growth} the path has at most $3W + 1$ children in the next layer.
    Due to the tie-breaking, the right-most of these children does not belong to the same large tile, yielding at most $3W$ neighbors in the next layer inside the large tile.
    Iterating this argument yields that the large tile contains at most $\sum_{i = 0}^{H - 1} 3^i \cdot W = (3^H - 1) / 2 \cdot W \in O(3^H \cdot W)$ small tiles.

    The lemma directly follows if the merging of triples only merges disjoint triples.
    Note that this is the case if each large tile in the $(k - 1)$-st ring is adjacent to at least four tiles in the $k$-th ring.
    Consider a large tile in the $(k - 1)$-st ring.
    As in the argument above, it has $W$ vertices in the path in the bottom layer.
    Due to \cref{lem:3_7_tiling_exponential_growth}, these have at least $2W + 1$ children in the next layer, at least $2W$ of which belong to the same tile (again, we loose the $+1$ as the right-most child belongs to the neighboring large tile).
    Thus, in its top layer, the large tile has at least $2^{H - 1}\cdot W$ vertices, yielding at least $2^H \cdot W$ neighbors in the bottom layer of the $k$-th ring.
    Separating them into paths of up to $2W$ vertices split them into at least $2^{H - 1}$ large tiles.
    As $H \ge 3$, this yields at least four adjacent large tiles, which concludes the proof.
\end{proof}

It remains to lower-bound the distance between non-adjacent large tiles.
As arguing about geometric distances between points in tiles of the $\{7, 3\}$-tiling is rather tedious, it is more convenient to translate between these geometric distances and graph-theoretic distances of the dual $\{3, 5\}$-tiling.
The following lemma justifies this.

\begin{lemma}\label{lem:geometric_graph_distances}
    Consider a regular $ \{ 7, 3 \} $-tiling in $ \Hyp^2 $ and let $ D $ be its dual.
    Then for each two vertices of the dual whose (graph-theoretic) distance in $ D $ is at least $d \geq 1 $, the (geometric) distance between any two points in $ \Hyp^2 $ lying in the corresponding tiles is at least $ (d - 1) / 2 $.
\end{lemma}

\begin{proof}
    Let $x$ and $y$ be two points in the hyperbolic plane.
    Let $v_x$ be the tile of the $\{7, 3\}$-tiling that contains $x$ (we assume without loss of generality that $x$ lies in the interior of a tile).
    Analogously, let $v_y$ be the tile containing $y$.
    We need to show that if the graph-theoretic distance between $v_x$ and $v_y$ in the dual $D$ is at least $d$, then the geometric distance between $x$ and $y$ is at least $(d - 1) / 2$.
    We show the contraposition, i.e., if the geometric distance between $x$ and $y$ is at most $(d - 1) / 2$, then the graph-theoretic distance between $x$ and $y$ is at most $d$.

    We show this by induction over $d$.
    For $d = 1$, the claim is obviously true.
    For larger $d$ let $x'$ be the point on the segment between $x$ and $y$ with distance $1/2$ from $x$ and let $v_x'$ be the corresponding tile.
    First observe moving form $x$ to $x'$ shrinks the geometry distance to $y$ by $1/2$, i.e., the geometric distance between $x'$ and $y$ is at most $(d - 2) / 2$.
    Thus, by induction, the graph-theoretic distance between $v_x'$ and $v_y$ in $D$ is at most $d - 1$.
    It remains to show that $v_x'$ is a neighbor of $v_x$ in $D$, i.e., that two points $x$ and $x'$ with geometric distance $1/2$ are either in the same or in neighboring tiles.
    For this, we note that the side length of the regular $7$-gons in a $\{7, 3\}$-tiling is roughly $0.56$\footnote{For the side length $\ell$, it holds that $\cosh(\ell / 2) = \cos(\pi / 7) / \sin(\pi / 3)$.}.
    Thus, even if $x$ lies at the corner of its tile, a point $x'$ at distance $1/2$ must lie in the same or an adjacent tile.
\end{proof}

Due to the connection in \cref{lem:geometric_graph_distances}, we can now give lower bounds for geometric distances between tiles in the $\{7, 3\}$-tiling by giving lower bounds for graph-theoretic distances in its dual $\{3, 7\}$-tiling $D$.
When referring to distances in the following, we always mean the graph-theoretic distance.
The following lemma provides a canonical way of how shortest paths look like in $D$, which helps arguing about distances in~$D$.

\begin{lemma}
    \label{lem:3_7_tiling_shorest_paths}
    Let $u$ and $v$ be two vertices of the $\{3, 7\}$ tiling.
    If there is a vertical $u$-$v$-path, then it is a shortest path.
    Otherwise, there is a shortest $u$-$v$-path of the form $u, \dots, u', \dots, v', \dots, v$ such that the subpaths $u, \dots, u'$ and $v', \dots, v$ are vertical and $u', \dots, v'$ is horizontal of length at most $2$.
    Moreover, $u'$ is the right-most ancestor of $u$ and $v'$ the left-most ancestor of $v$ or vice versa.
\end{lemma}
\begin{proof}
    Note that if there is vertical $u$-$v$-path, then this is obviously a shortest path.
    Now consider a shortest $u$-$v$-path $P$ such that the sum of layers of vertices in $P$ is minimal among all shortest $u$-$v$-paths.
    Assume we move from $u$ to $v$ along $P$ and observe how the layer changes.
    This yields three types of edges, where the layer either increases, decreases, or stays the same; see Figure~\ref{fig:3_7_tiling_shortest_path} (left).
    Thus, for two consecutive edges, there are 9 combinations of edge types.
    We want to show that $P$ is the concatenation of three subpaths $P_1$, $P_2$, $P_3$ with $P_1$ and $P_3$ being vertical and $P_2$ being horizontal.
    Note that this is the case if and only if every pair of consecutive edges in $P$ are of the same type (3 combinations), a change from decrease to same (1 combinations), a change from same to increase (1 combination), or, if $P_2$ is just one vertex, a change from decrease to increase (1 combination).
    To show that the other three combinations cannot occur, let $xy$ and $yz$ be the two edges.
    For (increase, decrease), there clearly is an edge $xz$ as $x$ and $z$ are both parents of $y$, which means that they are siblings; see Figure~\ref{fig:3_7_tiling_shortest_path} (center).
    Thus, this cannot happen as $P$ is a shortest path.
    For (increase, same), assume without loss of generality that $z$ is the right sibling of $y$.
    Then $x$ is either directly connected to $z$ (which again contradicts minimality of $P$), or the right sibling $y'$ of $x$ is adjacent to $z$; again see Figure~\ref{fig:3_7_tiling_shortest_path} (center).
    Thus, replacing the two edges with $xy'$ and $y'z$ yields a path so the same length with lower layer sum; a contradiction to the assumption that this sum is minimal in $P$.

    \begin{figure}
        \centering
        \includegraphics{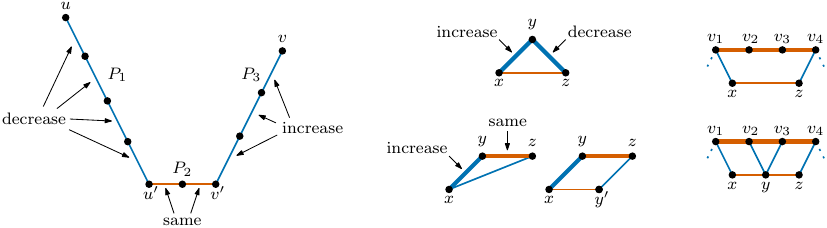}
        \caption{Illustrations for \cref{lem:3_7_tiling_shorest_paths}.
          Left: Partition of $P$ into vertical paths $P_1$, $P_3$ and the horizontal path $P_2$.
          Center: Combinations of edges that cannot occur on the path $P$, which is highlighted in bold.
          Right: $P_2$ has length at most $2$.}
        \label{fig:3_7_tiling_shortest_path}
    \end{figure}

    We next show that $P_2$ has length most $2$; see Figure~\ref{fig:3_7_tiling_shortest_path} (right).
    Assume for contradiction that $P_2$ is $u' = v_1, \dots, v_k = v'$ for $k > 3$ and assume without loss of generality that $P_2$ goes from left to right.
    Let $x$ be the only or right parent of $v_1$ and let $z$ be the only or left parent of $v_4$.
    If $x$ and $z$ are siblings, then we can replace $v_1, \dots, v_4$ with $v_1, x, z, v_4$ and obtain a path of the same length with lower layer sum.
    Thus, there has to be a vertex $y$ between $x$ and $z$ on the same layer.
    However, $y$ can only have the children $v_2$ and $v_3$ as it would otherwise be the right parent of $v_1$ or the left parent of $v_4$; a contradiction to the fact that every vertex as $3$ or $4$ children.

    It remains to show that $u'$ ($v'$) are the left-most (right-most) ancestors of $u$ ($v$) or vice versa.
    Assume without loss of generality that $P_2$ goes from left to right (or if $u' = v'$, then the child of $u'$ in $P_1$ is to the left its child in $P_3$).
    In this case, if $u'$ is not the right-most ancestor of $u$, we can replace the path $P_1$ by iteratively going from $u$ to its right (or only) child until it hits the layer of $u'$ or the path $P_3$ to obtain a shorter path.
    Analogously, it follows that $v'$ is the left-most ancestor of $v$.
\end{proof}

Observe that unless the shortest path in \cref{lem:3_7_tiling_shorest_paths} goes through the root, it provides a clear direction, which lets us use the terms left and right to describe the relative position of arbitrary vertices.

Also observe that the shortest paths behave similar to paths in a rooted tree in that we basically go up to something like a lowest common ancestor.
Except here, the parent is not always unique and the lowest common ancestor forms a path instead of a single vertex.
The next lemma strengthens this similarity to trees by showing that sometimes having two parents does not lead to a much larger number of ancestors.

\begin{lemma}
    \label{lem:3_7_tiling_at_most_two_ancestors}
    Every vertex of the $\{3, 7\}$-tiling has a most two ancestors in each layer.
\end{lemma}
\begin{proof}
    Let $v$ be a vertex of the $\{3, 7\}$-tiling.
    Note that we reach the left-most ancestor of $v$ on a different layer by always choosing the left parent, when there are two parents to choose from.
    Let $P_\ell$ be the resulting path.
    Analogously, let $P_r$ be the path reaching the right-most ancestor on the same layer.
    We show that $P_\ell$ and $P_r$ do not diverge too much, i.e., on a fixed layer, they either use the same vertex or siblings.
    As both paths start in $v$, this invariant is true initially.
    Now consider one of the later steps from one layer to the next.
    If $P_\ell$ and $P_r$ share the same vertex $u$ in the previous layer, then $u$ either has one parent or it has two parents that are siblings.
    Both cases maintain the invariant.
    Otherwise, $P_\ell$ and $P_r$ use two different vertices $u_\ell$ and $u_r$ that are siblings.
    As each face is a triangle, $u_\ell$ and $u_r$ need to have a common parent $w$ on the next layer.
    If both of them had an additional parent, one to the left and one to the right of $w$, then $w$ would have only two children.
    As every vertex including $w$ has at least three children, one of the paths $P_\ell$ or $P_r$ includes $w$ and the other either also goes through $w$ or through its left or right sibling.
\end{proof}

Now we are ready to show that pairs of non-adjacent large tiles are far apart.
We start with large tiles separated by a ring, continue with large tiles from the same ring, and finally consider tiles in adjacent rings.
Recall that we consider graph-theoretic distances, i.e., the distance between large tiles refers to the minimum graph-theoretic distance between included small tiles in the dual $D$.

\begin{lemma}
    \label{lem:tiling_large_r_non_adj_rings}
    Two large tiles from non-adjacent rings have distance at least $H + 1$.
\end{lemma}
\begin{proof}
    Clearly any path from between tiles in non-adjacent rings contains at least $H$ vertices from a ring separating them plus tow additional vertices for the first and last vertex.
    Thus, the shortest path contains at least $H + 2$ vertices, yielding a distance of at least $H + 1$.
\end{proof}

We continue with the case of non-adjacent tiles in the same ring.

\begin{lemma}
    \label{lem:tiling_large_r_same_ring}
    Two non-adjacent large tiles from the same ring have distance at least $2\log_3(W) - 2$.
\end{lemma}
\begin{proof}
    Consider two vertices $u$ and $v$ from the non-adjacent tiles.
    As they are non-adjacent, there is a large tile $S$ between them.
    Assume without loss of generality that $u$ lies to the left of $S$ and $v$ lies to the right of $v$.
    We first argue that we essentially only have to consider the case where $u$ and $v$ lie on the smallest layer of the considered ring.
    For this, let $P_S$ be the path in this smallest layer that defined $S$.
    Recall that due to the definition of the large tile $S$, the right-most ancestors of $u$ and $v$ in the layer of $P_S$ are to the left and right of $P_S$, respectively.
    Note that the left-most ancestor of $v$ might lie on $P_S$, but then it is the last vertex on $P_S$ due to \cref{lem:3_7_tiling_at_most_two_ancestors}.

    Thus, when considering the shortest $u$-$v$-path as described in \cref{lem:3_7_tiling_shorest_paths}, the vertical path from $u$ to its ancestor $u'$ cannot enter the separating tile $S$.
    Moreover, the vertical path from $v$ to its left-most ancestor $v'$ might enter $S$ but only barely.
    In any case, the distance between $u$ and $v$ is strictly larger than the distance between the first and last vertex of $P_S$.
    By construction, $P_S$ contains $W$ vertices.
    Thus it remains to give a lower bound on the distance between two vertices on the same layer based on the length of the horizontal path.

    For this, let $u_0 = u$ and $v_0 = v$ be the vertices we start with on the same layer $\ell$ and let $u_i$ and $v_i$ be their ancestors on layer $\ell - i$ in the shortest path $u$-$v$-path as described by \cref{lem:3_7_tiling_shorest_paths}.
    Let $W_i$ be the number of vertices on the horizontal path between $u_i$ and $v_i$ in layer $\ell - i$.
    Clearly $W_0 = W$.
    Moreover, due to \cref{lem:3_7_tiling_shorest_paths}, we have to decrease the layer until $W_i \le 3$ (note that the lemma refers to the length of $2$ for the horizontal path, while we are here interested in the number of vertices, which is one larger).
    Now consider how $W_i$ changes when going from layer $\ell - i$ to layer $\ell - (i + 1)$.
    Now, due to \cref{lem:3_7_tiling_exponential_growth}, we get that $W_i \le 3 \cdot W_{i + 1} + 1$ and thus $W_{i + 1} \ge W_i / 3 - 1/3$.
    From this, we get $W_i \ge W_0 / 3^i - \sum_{j = 1}^i 1/3^j$.
    The geometric sum is upper bounded by $1/2$ and $W_0 = W$, yielding $W_i \ge W / 3^i - 1/2$.
    Let $\hat\imath$ be the largest $i$ we see in the shortest path.
    As mentioned above, we then have $W_{\hat\imath} \le 3$.
    Thus, $3 \ge W / 3^{\hat\imath} - 1/2$, which resolves to $\hat\imath \ge \log_3(W) - \log_3(3 + 1/2)$.
    It follows that the distance between $u$ and $v$ is at least $2\hat\imath \ge 2 \log_3(W) - 3$, where the $3$ comes from $\log_3(3 + 1/2) \le 3/2$.  As the distance between the non-adjacent tiles is strictly larger than the distance between $u$ and $v$, we get the claimed bound. 
\end{proof}

Finally, we consider tiles from adjacent rings.

\begin{lemma}
    \label{lem:tiling_large_r_adj_rings}
    Two non-adjacent large tiles from adjacent rings have distance at least $\log_3(W) - 1$.
\end{lemma}
\begin{proof}
    Let $t_1$ and $t_2$ be the two large tiles of adjacent rings $R_1$ and $R_2$, where $R_1$ denotes the smaller ring; see \cref{fig:adjacent_rings}.
    Assume without loss of generality that $t_1$ lies to the left of $t_2$.
    Moreover, from the original large tiles (i.e., before the merging of triples) let $t_1'$ be the right-most tile in $R_2$ that is still adjacent to $t_1$.
    Recall that $t_1'$ has been merged with its two neighbors, which we call $t_\ell'$ and $t_r'$, i.e., $t_\ell'$, $t_1'$, and $t_r'$ together form one large tile.
    Note that, $t_1'$ and $t_2$ are from the same ring and are separated by $t_r'$.
    Thus, by \cref{lem:tiling_large_r_same_ring}, $t_1'$ and $t_2$ have distance at least $2\log_3(W) - 2$.\footnote{Technically, $t_1'$ is no longer a large tile because it has been merged with $t_\ell'$ and $t_r'$.
      However, \cref{lem:tiling_large_r_same_ring} works independently of the merging of triples, so we can apply it here to the original large tiles.}

    Now consider vertices $u \in t_1$ and $v \in t_2$ with minimum distance $d$.
    By construction, $u$ has a descendant $u'$ in the bottom layer of $t_1'$, which clearly has distance at most $d$ form $u$.
    This yields a path of length $2d$ from $u' \in t_1'$ to $v \in t_2$.
    As argued above, such a path has length at least $2\log_3(W) -3$ and thus $d \ge \log_3(W) - 1$ as claimed.
\end{proof}
We are now ready to prove \cref{lem:tiling} for $ r \geq c \approx 0.62 $, which we restate here for convenience.

\begin{samepage}
\begin{lemma}[\cref{lem:tiling} for $ \mathsf{r \geq c} $]\label{lem:tiling_large_radius}
    \tilingLemma{$ r \geq c $}[$ O(3^{8r}) $]
\end{lemma}
\end{samepage}

\begin{proof}[Proof of \cref{lem:tiling} for $ \bm{r \geq c} $]
    We use the tiling as described above and choose appropriate values for $W$ and $H$.
    We need to make sure that any two points in non-adjacent large tiles have geometric distance larger than $2r$.
    Due to \cref{lem:geometric_graph_distances}, it is sufficient to show that the graph-theoretic distance is lager than $4r + 1$.
    Let $d = \lfloor 4r + 2\rfloor$ be the graph theoretic distance we aim for.

    Due to \cref{lem:tiling_large_r_non_adj_rings}, setting $H \ge d$ makes sure that large tiles from non-adjacent rings have graph-theoretic distance at least $d$.
    For non-adjacent large tiles in the same ring (\cref{lem:tiling_large_r_same_ring}) or adjacent rings (\cref{lem:tiling_large_r_adj_rings}), the graph-theoretic distance is at least $2\log_3(W) - 2$ and $\log_3(W) - 1$, respectively.
    The latter is clearly smaller and thus both are at least $d$ if $\log_3(W) - 1 \ge d$, which is equivalent to $W = 3^{d - 1}$.
    Thus, we set $H = d$ and $W = 3^{d - 1}$.
    As $H \ge 3$, we can apply \cref{lem:tiling_large_r_count_small_tiles} showing that each large tile contains $O(3^H \cdot W) = O(3^{2d}) = O(3^{8r})$, which is the desired bound.

    As a final note, we implicitly assumed in the construction that each ring contains enough vertices in its smallest layer such that we can split it using multiple paths of $W$ vertices.
    For our choice of $W$ and $H$, this is not necessarily the case for the first few rings.
    However, getting the whole ring as just one large tile until we have at least $2W$ vertices in the layer does clearly not change the asymptotics.
\end{proof}

\begin{figure}
    \centering
    \includegraphics[page=2]{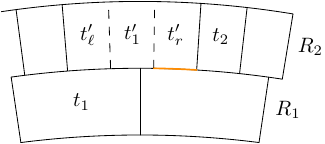}
    \caption{Two adjacent rings $ R_1, R_2 $ with large tiles $ t_1, t_2 $. Observe that $ t_1 $ and $ t_2 $ have smaller distance than $ t_1' $ and  $ t_2 $ (highlighted orange). The dashed lines indicate which tiles are merged.}
    \label{fig:adjacent_rings}
\end{figure}

To finish the subsection, it is now left to close the gap between $ d \approx 0.53 $ and $ c \approx 0.62 $.

\begin{proof}%
[Proof of \cref{lem:tiling} for $ \bm{r \geq d} $]
%
    For $ r \geq c \approx 0.62 $, \cref{lem:tiling_large_radius} gives the desired tiling.
    So, let $ r' \in [d, c] $ and consider the tiling provided by 
    \cref{lem:tiling_large_radius}.
    For this tiling, the second property is already stronger than what we need since $ 2r' \leq 2r $.
    For the first property, we are guaranteed that every tile can be covered by $ O(3^{8r}) $ disks of radius $ r = c $, but we have to use disks of radius $ r' \leq r $.
    For this, we show that a disk of radius $ c $ can be covered by constantly many disks of radius $ d $.
    Indeed, we have $ 2d \approx 1.06 > 0.62 \approx c $, so a disk of radius $ d $ touching the center of a disk of radius $ c $ leaves the larger disk and thus covers a circular sector with constant, positive angle $ \phi $.
    Hence, $ 2 \pi / \phi $ disks suffice to cover one disk of radius $ c $, and each tile can be covered by 
    $ O(3^{8r} \cdot 2 \pi / \phi) 
        = O(3^{8r}) 
        \subseteq O(3^{8(r' + c - d)}) 
        = O(3^{8 r'} \cdot 3^{c - d}) 
        = O(3^{8r'}) $ 
    disks of radius $ d \leq r' $.
\end{proof}


\subsection{Tiling for \texorpdfstring{$ \bm{r \leq d} $}{r < d}}
\label{sec:tiling_small_radius}

Having shown \cref{lem:tiling} for HUDGs with disk radii at least $ d \approx 0.53 $, it is left to deal with smaller radii.
Recall that $ d $ is chosen such that $ 2d $ is the distance between two opposite sides of a tile in a regular $ \{ 4, 5 \} $-tiling, which becomes important now.
We remark that even though row-treewidth is monotone under taking subgraphs, we cannot simply increase the radius to $ d $ and use the tiling from \cref{sec:tiling_large_radius} since this might also increase the clique size, possibly resulting in a super-constant clique size for the family.
We first prove a slightly weaker lemma in which points of distance at most $ 2r $ may lie in tiles sharing only a vertex but not an edge, which allows the tiling to have vertices with larger degree than 3.
To compensate, we require that the maximum degree is not too large and that the tiles are 4-gons that expand in a favorable way.
For this, we call a tiling \emph{concentric} if its vertices can be partitioned into $ C_1 $ containing a single vertex called the \emph{center} and cycles $ C_2, C_3, \dots $ such that each $ C_i $ is in the interior of $ C_{i + 1} $ and each vertex in $ C_i $ has a neighbor in $ C_{i + 1} $.
We call such a partition a \emph{concentric partition}.
Observe that this also induces a concentric partition of the dual, except for the innermost part being a facial cycle instead of a single vertex.

\begin{lemma}\label{lem:tiling_small_radius}
    For every $ r \in (0, d) $, there is an (irregular) concentric tiling of the hyperbolic plane with maximum degree $ 5 $ consisting of $ 4 $-gons such that
    \begin{itemize}
        \item every tile can be covered with $ O(1) $ disks of radius $ r $ and
        \item each two points with distance at most $ 2r $ lie in the same tile or in two tiles sharing a vertex.
    \end{itemize}

\end{lemma}


\begin{figure}
    \centering
    \includegraphics{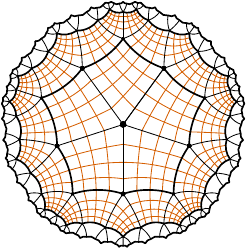}
    \caption{%
        A $ \{4, 5\} $-tiling (black) that is subdivided (red) so that each small tile can be covered by $ O(1) $ disks of radius $ r $.
    }
    \label{fig:subdivide_tilings_overview}
\end{figure}

\begin{proof}
    First observe that a $ \{ 4, 5 \} $-tiling satisfies the second property for all $ r < d $ since the distance between any two points in distinct tiles not sharing a vertex is at least as large as two points on opposite sides of a tile, hence their distance is at least $ 2d > 2r $.
    To also fulfill the first property, we subdivide the tiles so that they become small enough to be covered by a constant number of disks with radius $ r $, while keeping them large enough such that they cannot be crossed by a segment of length $ 2r $, i.e., no segment of this length intersects two opposite sides of a tile.
    
    \begin{figure}
        \centering
        \includegraphics[page = 1]{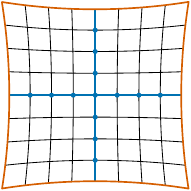}
        \hspace{4em}
        \includegraphics[page = 3]{subdivision}
        \hspace{1em}
        \includegraphics[page = 4]{subdivision}
        \caption{%
          Left: 8-subdivision of a regular 4-gon. The subdivision points are marked blue.
          Center: The 4-gon used for \cref{lem:tiling_small_radius}.
          Right: Illustration of the Saccheri quadrilateral from \cref{lem:tiling_small_radius}.
        }
        \label{fig:subdivision}
    \end{figure}
    
    Let us first define and analyze the subdivision before we apply it to a specific value $ r $.
    For $ k \geq 1 $, we define a $ k $-subdivision of a regular 4-gon $ S $ in a $ \{ 4, 5 \} $-tiling as follows, see also \cref{fig:subdivide_tilings_overview,fig:subdivision} (left).
    For a pair of opposite sides, consider the unique common perpendicular segment $ p $.
    We remark that $ p $ exists and is unique by the ultraparallel theorem, and since $ S $ is regular, $ p $ contains the center of $ S $ and has length $ 2d $.
    We subdivide $ p $ into $ k $ segments of length $ 2d / k $ each, and call the $ k - 1 $ points that are endpoints of two such segments the \emph{subdivision points}.
    Now we take the perpendiculars to $ p $ at each subdivision point and call their restrictions to $ S $ the \emph{subdivision segments}.
    The $ k $-subdivision of $ S $ is then obtained by subdividing at each subdivision segment, for both pairs of opposite sides of $ S $, and consists of $ k^2 $ 4-gons, which we call \emph{small tiles}.
    In particular, the $ 1 $-subdivision is $ S $ itself.
    
    Next, we upper- and lower-bound the side lengths of the small tiles, where the upper bound is needed for the first property and the lower bound by the second.
    We remark that the sides get longer the closer a small tile is to a vertex of $ S $, while they are shorter close to $ p $. 
    The lower bound of $ 2d / k $ is given by the distance between two consecutive subdivision points.
    Indeed, since the segment between two subdivision points is perpendicular to the subdivision segments, it is the shortest segment between them.
    
%

    For the upper bound, let $ \ell $ denote the length of a longest side of a small tile and let $ s = 2d / k $ denote the distance between two consecutive subdivision points.
    We aim to bound $ \ell $ in terms of $ s $.
    We note that if $k = 1$, then $ s = 2d $ and $ \ell $ is just the side length of the 4-gon of the $\{4, 5\}$-tiling and we get a bound from \cref{lem:distances-in-tilings}.
    However, the following more general argument holds for any $k \ge 1$.
    To bound $\ell$, consider the 4-gon $S'$ that is bounded by some side of length $ \ell $, by two parallel subdivision segments, and by the perpendicular $ p $ to the two subdivision segments.
    We refer to \cref{fig:subdivision} (center) for an illustration.

    We note that $S'$ is almost a Saccheri quadrilateral, i.e., a quadrilateral that has two equal sides perpendicular to its base, and a forth side opposite the base called \emph{summit}. 
    Here, our base is $ s $ and we have to perpendicular legs, but they do not have exactly the same length.
    One leg has length $ d' \approx 0.62 $, where $ 2d' $ is the side length of a tile in a $ \{ 4, 5\} $-tiling; see \cref{lem:distances-in-tilings}.
    The other leg is slightly shorter with length between $d'$ and $d$.
    To bound the length of $\ell$, we do the following.
    We mirror $S'$ along the leg of length $d'$ yielding a Saccheri quadrilateral with base of length $2s$ and two legs of length at most $d'$.
    Let $\ell'$ be the length of its summit opposite to the base; see \cref{fig:subdivision}.
    Note that this summit of length $\ell'$ also forms an isosceles triangle with the two sides of length $\ell$ with an angle of $2 \cdot 2\pi / 5 = 4 \pi / 5$ between them.
    We now do two estimations.
    First as this angle in the triangle is rather large, we see that $\ell'$ is roughly twice $\ell$.
    Thus, instead of bounding $\ell$ in terms of $s$, we bound $\ell'$.
    And second, we use the standard formula for Saccheri quadrilaterals to bound the length $\ell'$ of the summit by the length of the base.

    For the triangle, note that it can be split into two right triangles with hypotenuse $\ell$, side $\ell'/2$ and opposite angle $2\pi / 5$.
    Using trigonometry of right angles, we get
    \begin{equation*}
        \sin \frac{2 \pi}{5} = \frac{\sinh \frac{\ell'}{2} }{\sinh \ell}
        \quad \iff \quad
        \sinh \ell = \frac{\sinh \frac{\ell'}{2} }{\sin \frac{2 \pi}{5}}
    \end{equation*}
    It follows from the definition of $\sinh$ that $\ell \le \sinh \ell$, so we aim to upper-bound $ \sinh(\ell'/2) $.
    For this, we use that $ \sinh(x) \leq x \cdot \sinh(a) / a $  holds for all $ x \leq a $, which we apply for $ a = 0.63 $.
    Then, as $\ell' / 2 \le \ell \le d' < 0.63$ we obtain $\sinh(\ell' / 2) \le \ell' / 2 \cdot \sinh(0.63) / 0.63$.  Thus, we have
    \begin{equation*}
        \ell \le
        \sinh(\ell) =
        \frac{\sinh \frac{\ell'}{2} }{\sin \frac{2 \pi}{5}} \le
        \frac{\sinh (0.63)}{2 \cdot 0.63 \cdot \sin \frac{2 \pi}{5}} \cdot \ell'
        < 0.57 \ell'.
    \end{equation*}

    To bound $\ell'$ depending on $s$, recall that we have a Saccheri quadrilateral with base $2s$, legs of length $\le d'$ and summit $\ell'$.
    Thus, we get
    \begin{equation*}
        \sinh \frac{\ell'}{2} \le \cosh d' \cdot \sinh s.
    \end{equation*}
    Now we can use the same type of approximations as for the triangle: $\ell' / 2 \le \sinh(\ell' / 2)$ and $s \le d < 0.54$ implies $\sinh s \le s \cdot \sinh(0.54) / 0.54$.  Thus, we have 
    \begin{equation*}
        \frac{\ell'}{2} \le
        \sinh \frac{\ell'}{2} \le
        \cosh(0.63) \cdot \frac{\sinh(0.54)}{0.54} \cdot s
        \quad \implies \quad
        \ell' < 2.53 s.
    \end{equation*}
    Putting $\ell < 0.57\ell'$ and $\ell' < 2.53 s$ together yields $\ell < 1.5 s$.

    To finish the proof, it is now left to choose a suitable $ k $ depending on $ r $.
    Recall that we aim for a tiling whose tiles are small enough to be covered with a constant number of disks of radius $ r $ and, at the same time, are large enough to separate tiles not sharing a vertex.
    To meet the second requirement, we choose the shortest distance $ s $ between two opposite sides of a small tile to be larger than $ 2r $.
    That is, we choose $ k $ maximal such that $ s = 2d / k > 2r $.
    Note that as long as $s$ is larger than $4r$, we continue subdividing so that $s$ is halved.
    Thus, choosing $k$ maximal such that $s > 2r$ implies $2r < s \le 4r$.

    To cover the resulting tiles with 256 disks of radius $ r $, consider the $ k' $-subdivision of the initial regular 4-gon with $ k' = 16 k $.
    Note that this refines the the $ k $-subdivision since we have $ (16 k)^2 $ small tiles, compared to $ k^2 $ in the $ k $-subdivision, i.e., each small tile of the $ k $-subdivision is further subdivided into 256 even smaller tiles.
    In the $ k' $-subdivision, the shortest distance between two opposite sides of a tile is $ s' = s / 16 \leq r / 4 $, so the length of a longest side is at most $ 1.5 s' < r / 2 $.
    With the triangle inequality, we conclude that the distance between any two points in a tile of the $ k' $-subdivision is at most $ r $.
    Hence, putting the center of a disk with radius $ r $ at any point in a tile covers the whole tile.
    Therefore, each of the 256 subtiles can be covered by a single disk, and thus each tile of the $ k $-subdivision can be covered by 256 disks of radius $ r $.

    To show that the tiling is concentric, we start by defining the cycles $C_1, C_2, \dots$.
    For $C_1$, we pick a vertex of degree five, i.e., a vertex of the original tiling.
    For $C_i$, consider the graph that remains after deleting the vertices form the previous cycles $C_1, \dots, C_{i - 1}$.
    Then, $C_i$ consists of the vertices incident to the face containing the previous cycles.
    To see that this gives concentric cycles, first consider the case that we do not subdivide, i.e., we have just the $\{4, 5\}$-tiling; also see \cref{fig:subdivision-concentric} (left).
    Due to the fact that all faces are quadrilaterals, each vertex of $C_2$ has either one or no edge to the center.
    Moreover, each vertex has two edges in the cycle $C_2$, which leaves two or three edges to the outside.
    As all faces are quadrilaterals, this invariant remains true for later layers.
    As each vertex has degree $5$, two neighbors in its cycle and at most one neighbor in the previous cycle, it must have at least two neighbors in the next cycle.

    \begin{figure}
        \centering
        \includegraphics[page = 5]{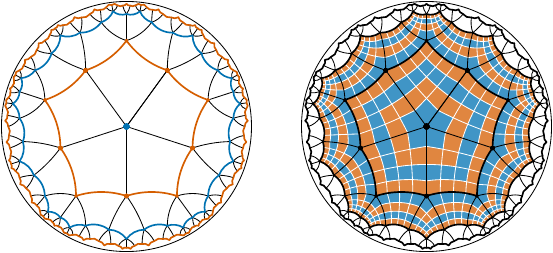}
        \caption{%
          Visualization of the tiling in \cref{lem:tiling_small_radius} being concentric.
          Left: The case where we do not subdivide.
          The cycles $C_1, C_2, \dots$ are colored alternatingly blue and red.
          Right: The case were the levels are subdivided.
          The cycles are at the transitions between a red and a blue regions.}
        \label{fig:subdivision-concentric}
    \end{figure}
    
    For the case where we subdivide, observe in \cref{fig:subdivision-concentric} (right) that each of the quadrilaterals from the original $\{4, 5\}$-tiling is subdivided into a normal $k \times k$-grid.
    The cycles $C_1, C_2, \dots$ now behave exactly the same as in the original $\{4, 5\}$-tiling, except that in between two of the previous cycles, we take $k$ steps to eat away the $k \times k$ grid of each of the original tiles.
\end{proof}

Note that the tiling constructed for \cref{lem:tiling_small_radius} becomes in a sense more Euclidean the further we subdivide.
For no subdivision, we have the regular $\{4, 5\}$-tiling where every vertex has degree $5$.
The subdivision now splits each tile into small tiles forming a normal Euclidean grid (except that the tiles are not quite squares).
Thus, the more we subdivide, the more Euclidean the tiling gets.
This is also reflected in the fact that the fraction of degree-$5$ vertices compared to degree-$4$ vertices shrinks.

We finish the proof of \cref{thm:upper_bound} by lifting \cref{lem:tiling_small_radius} to prove \cref{lem:tiling} for $ r \leq d $, that is, we aim for a tiling such that, first, each tile can be covered by $ O(1) $ disks of radius $ r $, and second, each two points with distance at most $ 2r $ are either in the same tile or in two tiles sharing an edge.
Note that the difference between the two lemmas is that for the second property, it does not suffice to share a vertex but the tiles are required to share an edge.

\begin{proof}[Proof of \cref{lem:tiling} for $ \bm{r \leq d} $]
    Consider a tiling from \cref{lem:tiling_small_radius} and note that the challenging situation with tiles sharing a vertex but not an edge occurs at vertices of degree 4 or 5.
    To solve this, we merge tiles so that the resulting tiling has maximum degree 3 and satisfies both properties of \cref{lem:tiling}.
    The first property is already met, so we only need to make sure to maintain it by not merging too many tiles.
    For the second property, observe that if at most three tiles meet at any point, then two tiles that share a vertex also share an edge.
    That is, we make sure that each two tiles sharing a vertex in the tiling of \cref{lem:tiling_small_radius} are merged or made adjacent, which guarantees the second property.
    

    \begin{figure}
        \centering
        \includegraphics{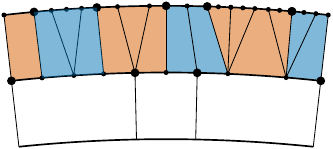}
        \caption{Merging tiles such that at most three large tiles meet at each vertex. The invariant applies to the thick vertices.}
        \label{fig:small_radius_invariant}
    \end{figure}

    Let $ C_1, C_2, \dots $ be the center and cycles of a concentric partition of the vertices of our tiling.
    Recall that by definition, each vertex in $ C_i $ has a neighbor in $ C_{i + 1} $, and thus each vertex has at least two incident tiles in the exterior.
    The base case consists of the up to five tiles incident to the center, which we merge into a single large tile, where a \emph{large tile} refers to a tile of the tiling we construct.
    Starting from here, we deal with the tiles cycle by cycle, maintaining the invariant that after handling all tiles in the interior of $ C_{i} $, the vertices in $ C_{i} $ that are already incident to two large tiles form an independent set.
    That is, the invariant controls those vertices we need to handle most carefully as they already have two incident large tiles, so they may get only one further.
    For $ C_2 $, the invariant is clearly satisfied as we only have one large tile so far.
    Now assume that the invariant holds for some $ C_i $, and we aim to deal with the incident tiles in the exterior of $ C_{i} $, which also induce a cycle.
    We refer to \cref{fig:small_radius_invariant} for an illustration.
    For each vertex that is already incident to two large tiles in the interior of $ C_{i} $, merge the two or three incident tiles in the exterior into a common large tile, yielding degree 3.
    Since each two of these vertices are non-consecutive in the cycle by our invariant, and each vertex in $ C_i $ has a neighbor in $ C_{i + 1} $, the resulting large tiles are pairwise distinct.
    The remaining tiles are split into chunks of three adjacent tiles wherever possible, and each chunk creates a new large tile.
    Since every vertex in $ C_i $ is incident to at most four tiles in the exterior of $ C_i $, this guarantees that each vertex is incident to at most two large tile in the exterior and at most three in total.
    For all tiles that are left, at least one of their neighboring tiles already belongs to some large tile.
    We let these tiles join any of their neighbors.
    Now the large tiles are made of two to five tiles, and thus vertices in $ C_{i + 1}$ that are incident to two large tiles are non-adjacent, satisfying our invariant.
\end{proof}


\section{Open Questions}\label{sec:open}

In this paper, we characterize for which functions $ r $ families of HUDGs with constant clique number and disk radius $ r $ admit product structure:
If $ r \in O(1) $, then every family has bounded row-treewidth and thus admits product structure, otherwise there are families with unbounded row-treewidth.
However, our upper bound is exponential in the disk radius, whereas our lower bound is only logarithmic. 
It is an interesting open question whether the exponential dependency is necessary, both for the tiling in \cref{lem:tiling} and for the row-treewidth of HUDGs.
We conjecture that it is indeed necessary for the tiling, but believe that the base can be improved, which would also improve the bound on the row-treewidth.
Note that even if there is an exponential lower bound for the tiling, it does not necessarily transfer to a lower bound on the row-treewidth of HUDGs.

\begin{question}
    What is the row-treewidth of HUDGs with clique number $ \omega $ and disk radius~$ r $?
\end{question}

With our lower-bound construction, we show that there are SHUDGs with radius $ \Theta(\log n) $ having treewidth $ \Omega( \omega \frac{\log n}{\log \log n} ) $.
This almost matches the upper bound of $ O(\omega \cdot \log n) $~\cite{Struc_Indep_Hyper_Unifor_Disk_Graph-Blaesius25,Treew_Graph_Balan_Separ-DvoraNorin19}.
We ask whether the remaining gap can be closed.
Moreover, as our construction provides a family of SHUDGs with clique number $\omega \in O(\log \log n)$, we ask whether such a bound can be achieved for every function $\omega(n)$; specifically for constant $\omega$.

\begin{question}
    Is there a family of $ n $-vertex HUDGs with clique number $ \omega $ and treewidth $ \Theta(\omega \cdot \log n) $?
    Is there a family with constant clique number and treewidth $\Theta(\log n)$?
\end{question}

Finally, we observe that our lower bound on the row-treewidth stops growing as soon as the disk radius is logarithmic in the number of vertices.
We conjecture that the reason is that for larger disk radii, there are no new families.

\todo[inline]{this only make sense for constant clique number: for $ \omega(n) = \log \log n $, of course, the row-treewidth stops increasing as it is upper-bounded by the treewidth. We maybe should discuss this somewhere.}

\begin{question}
    Are there families of HUDGs that do not admit a disk representation in the hyperbolic plane such that the disk radius is $ O(\log n) $?
\end{question}

\todo[inline]{Some more open questions: see discord / boxes below}


\Thomas{random question: Assume we have a family of graphs with $\omega = f(n)$ and $\mathrm{tw} = g(n)$ such that $f(n) \in o(g(n))$.
  Can I find an induced subgraph for each graph in the family such that the family of subgraphs has $\omega\in O(1)$ and $\mathrm{tw}\in \omega(1)$?
  This is probably too strong to be true, but I would also find it not completely unbelievable.}
  
\todo[inline]{%
    For non-induced subgraphs, this is true by the grid minor theorem.
    There are also grid minor theorems for induced minors, e.g., for bounded degree graphs and for $ H $-minor-free graphs.
    For theses graph classes, the question also has a positive answer:
    Although the branch sets do not only contain trees here, there is never a reason far a large clique in a branch set as a grid has bounded degree.
    However, our family does not belong to either of the two classes.
    
    Induced grid minor theorems need to be less general than the non-induced version as complete graphs have large treewidth but do not contain large induced grids, but I am not arware of any other counterexample. 
    \url{https://doi.org/10.1016/j.jctb.2023.01.002} (2023) seems to conjecture that sparse graph classes admit induced grid minor theorems (not sure how to interpret \enquote{one could expect}).
    I would even ask for something stronger:
    
    1. Is there a function $ f $ such that every graph with treewidth $ t $ and clique number $ \omega $ contains an induced grid minor of size $ f(t/\omega) $?
    
    2. A positive answer to 1. implies a positive answer to the blue questions above. But is it also equivalent?
}

\bibliography{lipics-v2021-sample-article}

\end{document}